\documentclass[a4paper, 11pt, final]{article}

\usepackage{graphicx,psfrag}
\usepackage{amssymb,amsthm,amsmath}
\usepackage{enumerate}
\usepackage{amsfonts,enumerate}
\usepackage{a4wide}
\usepackage{hyperref}

\numberwithin{equation}{section}
\allowdisplaybreaks[1]
\newtheorem{proposition}{Proposition}[section]
\newtheorem{theorem}[proposition]{Theorem}
\newtheorem{lemma}[proposition]{Lemma}

\newtheorem{definition}[proposition]{Definition}

\newenvironment{proofof}[1]{\smallskip\noindent\emph{\textbf{Proof of #1.}}%
\hspace{1pt}}{\hspace{-5pt}{\nobreak\quad\nobreak\hfill\nobreak%
$\square$\vspace{8pt}\par}\smallskip\goodbreak}

\newcommand{\pint}[1]{\mathaccent23{#1}}
\newcommand{\Lloc}[1]{\mathbf{L^{#1}_{loc}}}
\newcommand{\C}[1]{\mathbf{C^{#1}}}

\newcommand{\Cc}[1]{\mathbf{C_c^{#1}}}

\newcommand{\modulo}[1]{{\left|#1\right|}}
\newcommand{\norma}[1]{{\left\|#1\right\|}}

\newcommand{\reali}{{\mathbb{R}}}
\newcommand{\interi}{{\mathbb{Z}}}
\newcommand{\naturali}{{\mathbb{N}}}
\newcommand{\tv}{\mathrm{TV}}
\newcommand{\BV}{\mathbf{BV}}
\newcommand{\Lip}{\mathbf{Lip}}
\newcommand{\spt}{\mathop{\mathrm{spt}}}
\newcommand{\sgn}{\mathop{\mathrm{sgn}}}

\DeclareMathOperator{\infess}{essinf}

\renewcommand{\epsilon}{\varepsilon}
\renewcommand{\phi}{\varphi}
\renewcommand{\theta}{\vartheta}
\renewcommand{\L}[1]{\mathbf{L^#1}}
\newcommand{\W}[1]{\mathbf{W^{#1}}}
\newcommand{\Wloc}[1]{\mathbf{W_{loc}^{#1}}}

\renewcommand{\d}[1]{\mathinner{\mathrm{d}{#1}}}
\newcommand{\caratt}[1]{\mathbf{\chi}_{\strut #1}}

\newcommand{\Caption}[1]{
  \begin{minipage}{0.75\linewidth}
    \caption{\footnotesize{#1}}
  \end{minipage}}

\title{A Traffic Model Aware of Real Time Data}

\author{Rinaldo M.~Colombo \\ \small Unit\`a INdAM \\
  \small Universit\`a degli Studi di Brescia \\ \small Via Branze, 38
  \\ \small 25123 Brescia, Italy \\ \small \texttt{Rinaldo.Colombo@UniBs.it} \\
  \and Francesca Marcellini \\ \small Dip.~di Matematica e Applicazioni \\
  \small Universit\`a di Milano -- Bicocca \\ \small Via Cozzi, 55 \\ \small 20125 Milano, Italy \\ \small
  \texttt{Francesca.Marcellini@UniMiB.it}}

\begin{document}

\maketitle

\begin{abstract}
  \noindent
  Nowadays, traffic monitoring systems have access to real time data,
  e.g.~through GPS devices. We propose a new traffic model able to
  take into account these data and, hence, able to describe the
  effects of unpredictable accidents. The well posedness of this model
  is proved and numerical integrations show qualitative features of
  the resulting solutions. As a further motivation for the use of real
  time data, we show that the inverse problem for the
  Lighthill--Whitam~\cite{LighthillWhitham} and
  Richards~\cite{Richards} (LWR) model is ill posed.

  \medskip
  \noindent\textit{2000~Mathematics Subject Classification:} 35L65,
  90B20 \medskip

\end{abstract}

\section{Introduction}
\label{sec:Intro}

Differently from fluid dynamics, traffic dynamics does not rely on
well established fundamental principles like the conservation of
momentum or energy. Apart from the conservation of the total number of
vehicles, the many traffic models available in the literature all have
to rely on assumptions on the drivers' behavior and these assumptions
always contain some arbitrariness.

On the other hand, present day measurement devices allow a detailed
knowledge of the traffic situation essentially in real time. This
leads to the possibility of improving models by means of real time
data. Here, we propose a model aware of real time data or, in other
words, that encodes these data. We stress that no deterministic model
whatsoever can predict the insurgence of an accident. On the other
hand, the present model is able to take into account such an event and
to describe its effects.

In the current literature, three different approaches are mainly used
to model traffic phenomena: microscopic, macroscopic and kinetic. For
an overview of vehicular traffic models at all scales, we refer to the
review papers~\cite{MR3253235, Bellomo, BressanCanicGaravelloHerty,
  Klar, Piccoli}. Here, we are concerned with \emph{macroscopic}
models, where traffic is described through the fraction $\rho = \rho
(t,x)$ of space occupied by vehicles at time $t$ and at position $x$.

\smallskip

From an analytic point of view, a justification of our insertion of
real time traffic data in the very formulation of the model is
provided by the difficulties inherent to the solution of the inverse
problem for a 1D scalar conservation law. Indeed, a rigorous approach
to the issue of finding the \emph{``right''} speed law $\mathcal{V} =
v (\rho)$ leads to an inverse problem that can be stated as
follows. Find the function $\mathcal{V} = v (\rho)$ so that the
solution to~\eqref{eq:2} best approximates the observed traffic
dynamics. More formally, we are led to consider the inverse problem,
see Figure~\ref{fig:problems}:
\begin{equation}
  \label{eq:1}
  \mbox{find } v \mbox{ so that }
  \int_0^T \modulo{\dot p(t) - v(\rho \left(t,p(t)+ \right)} \d{t}
  \mbox{ is minimal, where }
  \left\{
    \begin{array}{@{}l@{}}
      \partial_t \rho + \partial_x \left(\rho \, v (\rho)\right) =0
      \\
      \rho (0,x) = \rho_{0} (x)
    \end{array}
  \right.
\end{equation}
By means of an example, we show below that problem~\eqref{eq:1} is in
general ill posed.
\begin{figure}[!h]
  \centering
  \begin{psfrags}
    \psfrag{x}{$x$} \psfrag{rho}{$\rho $} \psfrag{t}{$t$}
    \psfrag{a}{$p(t)$}
    \includegraphics[width=3.8cm]{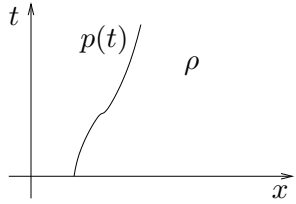}
  \end{psfrags}
  \caption{Situation considered by problem~(\ref{eq:1}). The
    trajectory $p = p (t)$ is measured, while the density $\rho = \rho
    (t,x)$ solves~\eqref{eq:2}.}
  \label{fig:problems}
\end{figure}
Moreover, a positive result in this direction is obtained, see
Proposition~\ref{prop:positive}, but it relies on assumptions that can
be hardly acceptable in a real situation.

\smallskip

The new macroscopic model we propose consists of a conservation law of
the type
\begin{equation}
  \label{eq:2}
  \partial_t \rho + \partial_x \left(\rho \, \mathcal{V} (t,x,\rho)\right) = 0
\end{equation}
where we propose to encode a measured trajectory $p = p (t)$ in the
speed law $\mathcal{V}$, obtaining
\begin{equation}
  \label{eq:3}
  \mathcal{V} (t,x,\rho)
  =
  \chi\left(x -p (t)\right) \, \frac{2\dot p (t) \, v (\rho)}{\dot p (t) + v (\rho)}
  +
  \left(1-\chi\left(x-p (t)\right)\right) v (\rho) \,.
\end{equation}
In other words, we use a time and space dependent speed law that
consists of an interpolation between measured data $p = p (t)$ and a
\emph{standard} speed law $v = v (\rho)$. The smooth function $\chi =
\chi (\xi)$ attains the value $1$ for $\modulo{\xi} \leq \ell$ and
vanishes when $\modulo\xi \geq L$, for two fixed constants $\ell, L$,
with $\ell < L$. Therefore, the speed law in~\eqref{eq:3} assigns the
measured speed $\dot p (t)$ near the vehicle providing the data,
i.e.~for $\modulo{x-p (t)} \leq \ell$, and assigns a \emph{standard}
speed $v\left(\rho (t,x)\right)$ far from the measuring vehicle,
i.e.~when $\modulo{x-p (t)} \geq L$. Related results on models based
on mixed microscopic--macroscopic descriptions are found
in~\cite{ColomboMarcelliniPreprint, LattanzioPiccoli, Marcellina}.

We prove existence, uniqueness and Lipschitz continuous dependence
from initial data of the solutions to the Cauchy problem
for~\eqref{eq:2}--\eqref{eq:3}. By means of a few numerical
integrations, we show below qualitative properties of the solutions
to~\eqref{eq:2}--\eqref{eq:3}. Remark that the extension to the case
of several measured trajectories is of a merely technical nature, both
at the analytic and at the numeric levels.

As a further remark, we observe that the use of real time data makes
the model \emph{``less falsifiable''}. On the other hand, we gain the
possibility of describing the effects of unpredictable accidents. The
current mathematical and engineering literatures show an increasing
interest in this direction, see also~\cite{1240476, ban2009optimal,
  Cheng2006particlefilter, WorkBlandinEtAl}.

\smallskip

The paper is organized as follows: the next section is devoted to the
inverse problem for a standard LWR~\cite{LighthillWhitham, Richards}
model. In Section~\ref{sec:Main} we study the Cauchy Problem for
system~\eqref{eq:3} and in Section~\ref{sec:A} we present some
numerical integrations of this model. All proofs are gathered in
the last section.

\smallskip

Throughout, we denote $\reali^+ = \left[0, +\infty\right[$ and
$\pint{\reali}^+ = \left]0, +\infty\right[$.  A Lipschitz constant of
the map $f$ is denoted $\Lip (f)$.  The maximal density (or
occupancy), i.e., the density at which vehicles are bumper to bumper
and can not move, is normalized to $1$.

\section{On an Inverse Problem for the LWR Model}
\label{sec:Inverse}

As a first step, we show that in general a solution to the inverse
problem~\eqref{eq:1} may fail to exist. Indeed, for $\epsilon \in [-1,
\, 1]$, consider the speed law and the corresponding flow
\begin{eqnarray}
  \label{eq:epsilon}
  v_\epsilon (\rho) = (1+\epsilon\rho) (1-\rho)
  \qquad \mbox{ and } \quad
  f_\epsilon (\rho) = (1+\epsilon\rho) (1-\rho) \rho
\end{eqnarray}
and choose the trajectory $p (t) = t/2$. Note that $f''_\epsilon
(\rho) < 0$ for all $\rho \in [0,1]$ and $\epsilon \in [-1/3,
1/3]$. For simplicity, we consider the Riemann Problem
\begin{equation}
  \label{eq:RP}
  \left\{
    \begin{array}{l}
      \partial_t \rho + \partial_x \left(\rho \, v_\epsilon (\rho)\right) = 0
      \\
      \rho (0,x) =
      \left\{
        \begin{array}{rrcl}
          1/8 & x & < & 0
          \\
          3/8 & x & > & 0
        \end{array}
      \right.
    \end{array}
  \right.
\end{equation}
whose solution is depicted in~Figure~\ref{fig:RP} in the two cases
$\epsilon <0$ and $\epsilon>0$.
\begin{figure}[htpb]
  \centering
  \begin{psfrags}
    \psfrag{x}{$x$} \psfrag{t}{$t$}
    \psfrag{rr}{$1/8$}\psfrag{r}{$3/8$}
    \includegraphics[width=4.8cm]{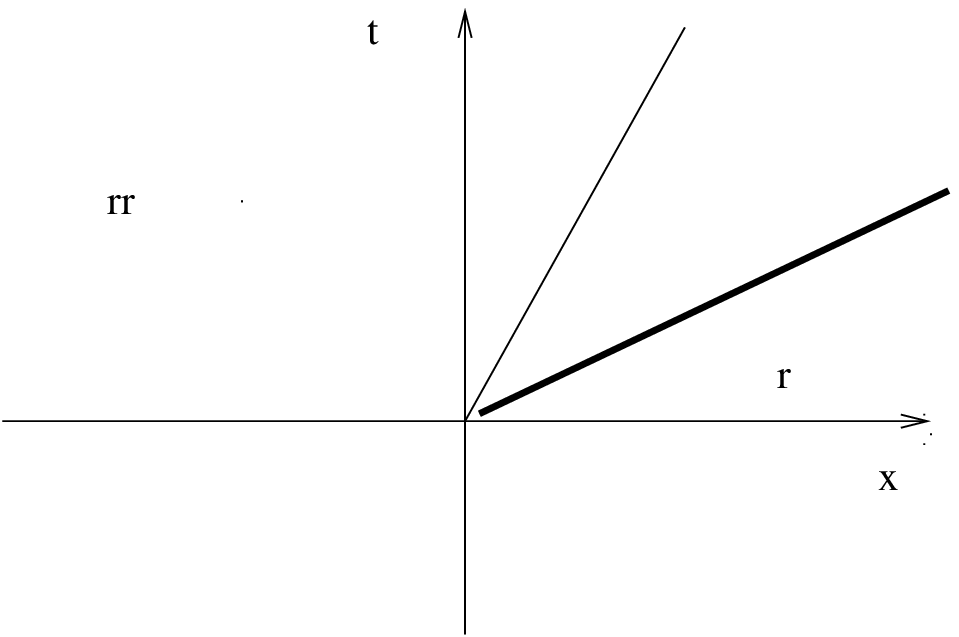}
  \end{psfrags}
  \hfil
  \begin{psfrags}
    \psfrag{x}{$x$} \psfrag{t}{$t$}
    \psfrag{u}{$1/8$}\psfrag{uu}{$3/8$}
    \includegraphics[width=4.8cm]{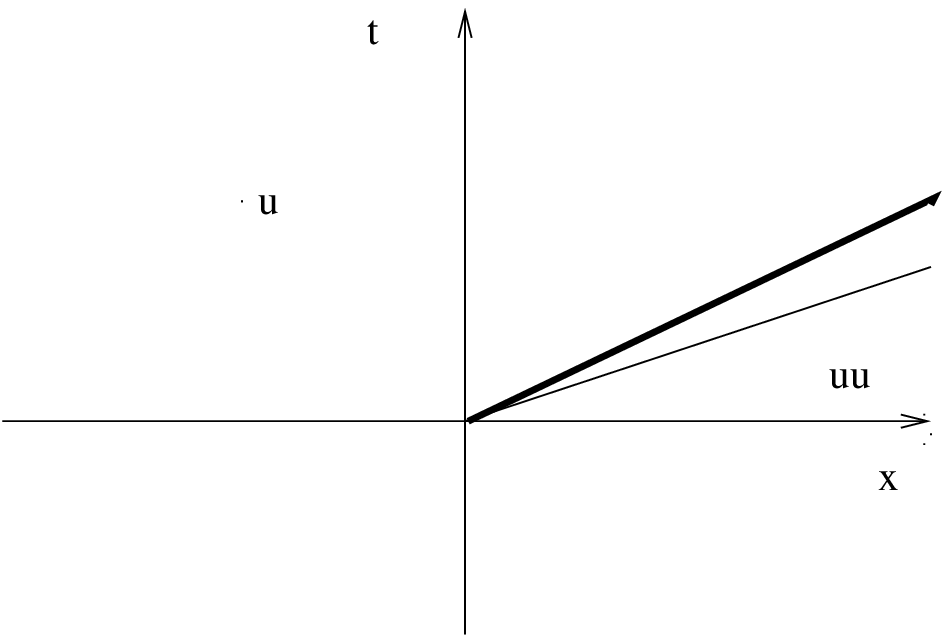}
  \end{psfrags}
  \Caption{Solution to the Riemann Problem~\eqref{eq:RP} with
    $v_{\epsilon}$ and $f_{\epsilon}$ as in~\eqref{eq:epsilon}. Left,
    the case $\epsilon<0$ and, right, $\epsilon>0$.}
  \label{fig:RP}
\end{figure}

The discrepancy between measured data and the description provided by
the LWR model can be estimated through straightforward computations as
follows:
\begin{displaymath}
  \int_0^T
  \modulo{
    \dot p (t)
    -
    v_\epsilon\left(\rho_\epsilon\left(t, p (t)\right)\right)}
  \d{t}
  =
  \left\{
    \begin{array}{l@{\quad\mbox{if }}r@{\,}c@{\,}l}
      \left(\frac{9}{8} + \frac{15\epsilon}{64}\right) T
      & \epsilon& \leq & 0
      \\
      \left(\frac{3}{8} + \frac{7\epsilon}{64}\right) T
      & \epsilon& > & 0
    \end{array}
  \right.
\end{displaymath}
This shows that the map
\begin{equation}
  \label{eq:phi}
  \begin{array}{rcl}
    \phi \colon [-1/3, \, 1/3] & \to & \reali
    \\
    \epsilon & \to &
    \displaystyle
    \int_0^T
    \modulo{
      \dot p (t)
      -
      v_\epsilon\left(\rho_\epsilon\left(t, p (t)+\right)\right)}
    \d{t}
  \end{array}
\end{equation}
does not attain a minimum, see Figure~\ref{fig:negative}. Clearly,
this example can be easily extended to more general, non constant,
functions $p$ and to more general Cauchy initial data.

\begin{figure}[htpb]
  \centering
  \begin{psfrags}
    \psfrag{e}{$\epsilon$} \psfrag{f}{$\phi$}
    \includegraphics[width=4.8cm]{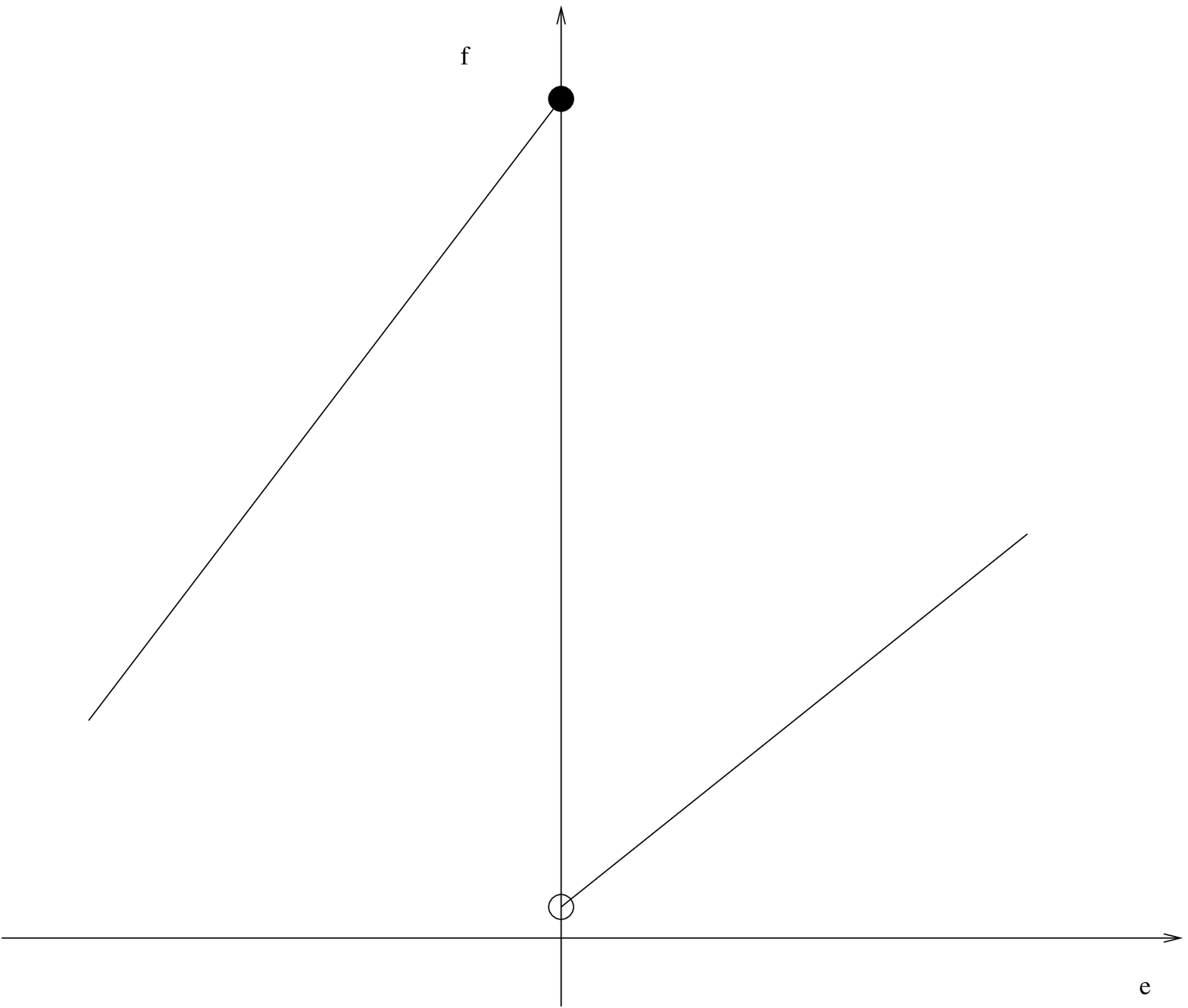}
  \end{psfrags}
  \Caption{The map $\phi$ defined in~\eqref{eq:phi} does not attain a
    minimum for $\epsilon \in [-1/3,\, 1/3]$. Above, $\phi(0+) = 3/8$
    and $\phi (0-) =9/8$.\label{fig:negative}}
\end{figure}

Due to the example above, one is lead to consider problem~\eqref{eq:1}
in a specific class of speed laws. A positive result is available in
the class of speed laws $v_{\strut V} (\rho) = V (1-\rho)$ with $V \in
[\check V, \hat V]$ being the maximal speed of cars along the
considered road. Clearly, we assume throughout that $\hat V > \check V
> 0$.

The following lemma plays a key role to obtain the basic continuity
estimate on the dependence of the error functional $\gamma \to
\int_0^T \modulo{\dot p (t) -v \left(\rho\left(t, \gamma
      (t)\right)\right)} \d{t}$ from a generic (non-characteristic)
curve $\gamma = \gamma (t)$.

\begin{lemma}
  \label{lem:1}
  Fix $T > 0$, $f \in \C1 (\reali; \reali)$ and $\rho_o \in (\L1 \cap
  \BV) (\reali; [0,1])$. Call $\rho$ the solution to
  \begin{equation}
    \label{eq:lem1}
    \left\{
      \begin{array}{l}
        \partial_t \rho + \partial_x f (\rho) = 0
        \\
        \rho (0,x) = \rho_o (x) \,.
      \end{array}
    \right.
  \end{equation}
  Choose two curves $\gamma_1, \gamma_2 \in \W{1,\infty} ([0,T];
  \reali)$, non characteristic in the sense that there exists a $c>0$
  such that
  \begin{displaymath}
    \dot \gamma_i (t) > f'\left(\rho (t, \gamma_i (t))\right) + c
    \qquad \mbox{ for all } t \in [0,T]
    \quad \mbox{ and } i=1,2 \,.
  \end{displaymath}
  Then,
  \begin{displaymath}
    \int_0^T
    \modulo{
      \rho\left(t, \gamma_1 (t)\right)
      -
      \rho\left(t, \gamma_2 (t)\right)}
    \d{t}
    \leq
    \frac{1}{c} \,
    \tv(\rho_o) \,
    \norma{\gamma_1 - \gamma_2}_{\C0 ([0,T]; \reali)} \,.
  \end{displaymath}
\end{lemma}

\noindent The proof is deferred to Section~\ref{sec:TD}.

\smallskip

We are now ready to prove the continuity result that implies the
existence of a solution to the inverse problem~\eqref{eq:1}.

\begin{proposition}
  \label{prop:positive}
  Let $T>0$, $\hat V, \check V$ be such that $\hat V > \check V >
  0$. Fix $\check \rho \in \left]0, 1 \right[$. If the initial datum
  $\rho_o \in (\L1 \cap \BV) (\reali; [0,1])$ and the path $p \in
  \W{1,\infty}$ are such that
  \begin{equation}
    \label{eq:brutte}
    \infess_{x \in \reali} \rho_o > \check \rho
    \quad \mbox{ and } \quad
    \infess_{t \in [0,T]} \dot p \geq \hat V (1-2\check\rho) \,,
  \end{equation}
  then the map
  \begin{displaymath}
    \begin{array}{@{}rcl@{}}
      \mathcal{E} \colon [\check V, \hat V] & \to & \reali
      \\
      V & \to & \displaystyle\int_0^T \modulo{ \dot p (t) - v_{\strut
          V}\left(\rho_{\strut V}\left(t, p (t)\right)\right)} \d{t}
    \end{array}
    \mbox{ where }
    \left\{
      \begin{array}{l}
        \partial_t \rho_{\strut V}
        +
        \partial_x \left(\rho_{\strut V} \, V \, (1-\rho_{\strut V})\right)
        =0
        \\
        \rho_{\strut V} (0,x) = \rho_o (x)
      \end{array}
    \right.
  \end{displaymath}
  is continuous.
\end{proposition}

\noindent The proof is deferred to Section~\ref{sec:TD}.

The existence of solution to the inverse problem~\eqref{eq:1} now
follows through a standard {Weierstra\ss} argument. Note however that
the two conditions in~\eqref{eq:brutte} do not agree with the needs of
a real application of this result. The former inequality requires the
vehicular density to be \emph{high}. At the same time, the latter
inequality in~\eqref{eq:brutte} imposes that the measured speed be
greater than $\hat V (1-\check \rho)$ uniformly in $\rho$, for $\rho
\in [\check \rho, 1]$. These two conditions are somewhat contradictory
to expecting that $\dot p$ is well approximated by $V (1-\rho)$ with
$V \in [\check V, \hat V]$.

\section{A Traffic Model Encoding Real Time Data}
\label{sec:Main}

This section is devoted to the Cauchy Problem for~\eqref{eq:3}, more
precisely:
\begin{equation}
  \label{eq:CP}
  \left\{
    \begin{array}{l}
      \displaystyle
      \partial_t \rho + \partial_x f (t,x,\rho) = 0
      \\
      \rho (0,x) = \rho_o (x)
    \end{array}
  \right.
\end{equation}
where $\rho_o \in \L\infty (\reali; [0,1])$ and
\begin{equation}
  \label{eq:CP1}
  \!\!\!  \!\!\!  \!\!\!
  \begin{array}{@{}r@{\;}c@{\;}@{}l}
    \displaystyle
    f (t,x,\rho)
    & = &
    \displaystyle
    \rho \, \mathcal{V} (t,x,\rho)
    \\
    \displaystyle
    \mathcal{V} (t,x,\rho)
    & = &
    \left\{
      \begin{array}{@{}ll@{}}
        \displaystyle
        \chi\left(x -p (t)\right)
        \frac{2 \, \dot p (t) \, v (\rho)}{\dot p (t) + v (\rho)}
        +
        \left[1-\chi\left(x-p (t)\right)\right] v (\rho)
        & \left(\dot p (t), v (\rho)\right) \neq (0,0)
        \\
        0 & \left(\dot p (t), v (\rho)\right) = (0,0)
      \end{array}
    \right.
    \!\!\!  \!\!\!  \!\!\!  \!\!\!
  \end{array}
\end{equation}
We posit below the following assumptions on the functions appearing
in~\eqref{eq:CP}--\eqref{eq:CP1}. Throughout, the maximal speed $V$ is
a fixed positive constant.
\begin{description}
\item[(v)] $v \in \C{0,1} ([0,1]; [0,V])$ is such that $\left\{
    \begin{array}{ll}
      v (1) = 0
      \\
      \rho \to v (\rho) & \mbox{ is non increasing,}
      \\
      \rho \to \rho \, v (\rho) & \mbox{ is strictly concave,}
      \\
      \rho \to \frac{\rho \, w \, v (\rho)}{w + v (\rho)}
      & \mbox{ is strictly concave, } \forall\,w > 0\,.
    \end{array} \right.$
\item[(p)] $p \in \C{1,1} (\reali^+; \reali)$ is such that $\dot p
  \geq 0$ for a.e.~$t \in \reali^+$.
\item[($\boldsymbol{\chi}$)] $\chi \in \Cc1 (\reali^+; [0,1])$.
\end{description}

\noindent Clearly, the usual choice $v (\rho) = V(1-\rho)$, for $V >
0$, satisfies~\textbf{(v)}. Moreover, as soon as $v \in \C2$, the
requirement that the map $\rho \to \frac{\rho \, w \, v (\rho)}{w
  + v (\rho)}$ be strictly concave follows from the other two
requirements on the function $v$, see Lemma~\ref{lem:conca} in
Section~\ref{sec:TD}. Condition~\textbf{(p)} simply states that the
acceleration of the measured trajectory is bounded and the speed has a
definite sign. The interpolating function $\chi$ needs only to be
sufficiently regular and to attain values in $[0,1]$, so that
$\mathcal{V}$ varies smoothly between $v (\rho)$, far from $p (t)$,
and the harmonic mean between $v (\rho)$ and $\dot p$, near to $p
(t)$.

The current literature, e.g.~\cite{Andreianov2000, KarlsenRisebro2003,
  Kruzkov, Panov_2010_2, Panov_2010_1}, offers different definitions
of solution to~\eqref{eq:CP}--\eqref{eq:CP1}.
Following~\cite{KarlsenRisebro2003, Kruzkov}, we first recall the
classical Kru\v zkov definition.

\begin{definition}[{\cite[Definition~1]{Kruzkov}}]
  \label{def:sol}
  Let $\rho_o \in \L\infty (\reali; [0,1])$. A map $\rho \in \L\infty
  (\reali^+ \times \reali; [0,1])$ is a \emph{Kru\v zkov solution}
  to~\eqref{eq:CP} if for any $k \in \reali$ and for any $\phi \in
  \Cc\infty (\pint{\reali}^+ \times \reali; \reali^+)$,
  \begin{equation}
    \label{eq:DefKruzkov}
    \begin{array}{l@{}l@{}l}
      \displaystyle
      \int_{\reali^+} \int_{\reali}
      &
      \displaystyle
      \Big[
      \modulo{\rho (t,x) - k} \, \partial_t \phi (t,x)
      \\
      &
      \displaystyle
      +
      \sgn\left(\rho (t,x) - k\right)
      \left(f\left(t, x, \rho (t,x)\right) - f (t,x,k)\right)
      \partial_x \phi (t,x)
      \\
      &
      \displaystyle
      -
      \sgn \left(\rho (t,x) - k\right)
      \partial_x f (t,x,k) \, \phi (t,x)
      \Big] \d{x} \d{t}
      &
      \geq 0
    \end{array}
  \end{equation}
  and there exists a set $\mathcal{E} \subset \reali$ of zero Lebesgue
  measure such that
  \begin{equation}
    \label{eq:DefKruzkov0}
    \lim_{t \to 0+, t \in [0,T] \setminus\mathcal{E}}
    \int_{\reali} \modulo{\rho (t,x) - \rho_o (x)} \d{x} = 0 \,.
  \end{equation}
\end{definition}

The weaker concept of solution proposed by Panov is of use in the
proofs below.

\begin{definition}[{\cite[Definition~3]{Panov_2010_1}}]
  \label{def:Panov}
  For any $k \in \reali$, call $\mu^k_c$, respectively $\mu^k_s$, the
  continuous, respectively singular, part of the distributional
  derivative $(t,x) \to \partial_x f (t,x,k)$ of the map $(t,x) \to f
  (t,x,k)$. A map $\rho \in \L\infty (\reali^+ \times \reali; [0,1])$
  is a \emph{Panov solution} to~\eqref{eq:CP}--\eqref{eq:CP1} if for
  any $k \in \reali$ and for any $\phi \in \Cc\infty (\reali^+ \times
  \reali; \reali^+)$,
  \begin{equation}
    \label{eq:DefPanov}
    \begin{array}{l@{}l@{}l}
      \displaystyle
      \int_{\reali^+} \int_{\reali}
      \Big[
      &
      \displaystyle
      \modulo{\rho (t,x) - k} \, \partial_t \phi (t,x)
      \\
      &
      \displaystyle
      +
      \sgn\left(\rho (t,x) - k\right)
      \left(f\left(t, x, \rho (t,x)\right) - f (t,x,k)\right)
      \partial_x \phi (t,x)
      \Big] \d{x} \d{t}
      \\
      \displaystyle
      -
      \int_{\reali^+} \int_{\reali}
      &
      \displaystyle
      \sgn \left(\rho (t,x) - k\right)
      \, \phi (t,x)
      \d{\mu^k_c(t,x)}
      +
      \int_{\reali^+} \int_{\reali}
      \phi (t,x) \d{\mu^k_s (t,x)}
      \\[12pt]
      \displaystyle
      +
      \int_{\reali^+}
      &
      \displaystyle
      \modulo{\rho_o (x) - k} \, \phi (0,x) \d{x}
      &
      \geq 0 \,.
    \end{array}
  \end{equation}
\end{definition}

The well posedness of~\eqref{eq:CP} is obtained through the following
propositions, whose proofs are detailed in
Section~\ref{sec:TD}. First, the existence of Panov solutions is
obtained.

\begin{proposition}
  \label{prop:exi}
  Let~\textbf{(v)}, \textbf{(p)} and~\textbf{($\boldsymbol{\chi}$)}
  hold. Let $\rho_o \in \L\infty (\reali; [0,1])$. Then,
  \cite[Theorem~2]{Panov_2010_1} applies, so that
  problem~\eqref{eq:CP}--\eqref{eq:CP1} admits a Panov solution in the
  sense of Definition~\ref{def:Panov}.
\end{proposition}

Then, we verify that in the case of~\eqref{eq:CP}, Pavov solutions are
also Kru\v zkov solutions.

\begin{proposition}
  \label{prop:def}
  Let~\textbf{(v)}, \textbf{(p)} and~\textbf{($\boldsymbol{\chi}$)}
  hold. Let $\rho_o \in \L\infty (\reali; [0,1])$. If $\rho$ is a
  Panov solution to~\eqref{eq:CP1}--\eqref{eq:CP1} in the sense of
  Definition~\ref{def:Panov}, then $\rho$ is also a Kru\v zkov
  solution to~\eqref{eq:CP}--\eqref{eq:CP1}, in the sense of
  Definition~\ref{def:sol}.
\end{proposition}

We are thus left with the task of proving the uniqueness of Kru\v zkov
solution and its stability with respect to the initial datum. Note
that~\cite[Theorem~1.1]{KarlsenRisebro2003} almost applies to the
present case, see Lemma~\ref{lem:Fregatura} for details.

\begin{proposition}
  \label{prop:Uni}
  Let~\textbf{(v)}, \textbf{(p)} and~\textbf{($\boldsymbol{\chi}$)}
  hold. Let $\rho_o \in \L\infty (\reali; [0,1])$. Then,
  problem~\eqref{eq:CP}--\eqref{eq:CP1} admits at most a unique Kru\v
  zkov solution in the sense of Definition~\ref{def:sol}. Moreover, if
  $\rho_o' \in \L\infty (\reali; [0,1])$ is another initial datum and
  $\rho' = \rho' (t,x)$ is the corresponding solution, the following
  Lipschitz estimate holds:
  \begin{equation}
    \label{eq:Lip}
    \norma{\rho' (t) - \rho (t)}_{\L1 (\reali; \reali)}
    \leq
    e^{C\,t} \norma{\rho_o' - \rho_o}_{\L1 (\reali; \reali)}
  \end{equation}
  where $C$ is defined in~\eqref{eq:C}, independent from $\rho_o$ and
  $\rho_o'$.
\end{proposition}

Together, the above propositions give our main result.

\begin{theorem}
  \label{thm:main}
  Let~\textbf{(v)}, \textbf{(p)} and~\textbf{($\boldsymbol{\chi}$)}
  hold. Then, for any $\rho_o \in \L\infty (\reali; [0,1])$,
  problem~\eqref{eq:CP}--\eqref{eq:CP1} admits a unique {Kru\v zkov}
  solution in the sense of Definition~\ref{def:sol} and the Lipschitz
  continuity estimate~\eqref{eq:Lip} holds.
\end{theorem}

\noindent The proof directly follows from propositions~\ref{prop:exi},
\ref{prop:def} and~\ref{prop:Uni}.

\section{Numerical Integrations}
\label{sec:A}

To numerically integrate the model~\eqref{eq:CP}--\eqref{eq:CP1} we
use the standard Lax-Friedrichs algorithm,
see~\cite[Section~12.1]{LeVeque}. Throughout, we use a fixed space
mesh size $\Delta x = 2.5\times 10^{-3}$.

As a first example, in Figure~\ref{fig:questa}, we choose the speed
law
\begin{equation}
  \label{eq:5}
  v (\rho ) = 1-\rho
\end{equation}
and the constant initial datum
\begin{equation}
  \label{eq:4}
  \rho_o (x) = 0.5 \,.
\end{equation}
Both at the analytic and at the numeric levels, the extension
of~\eqref{eq:CP}--\eqref{eq:CP1} to more than one measuring vehicle is
immediate. Here, we consider that two cars $p$ and $q$, which we
imagine equipped with a GPS measuring device, follow trajectories with
the same speed but exiting different initial position, say
\begin{equation}
  \label{eq:pqt}
  \dot p(t) = \dot q (t) =
  0.5 \, \caratt{[0,5[}(t)
  +
  0.6 \, \caratt{[5,6]}(t)
  +
  0.2 \, \caratt{[8,11]}(t)
  +
  0.4 \, \caratt{[13,18]}(t)\,,
  \qquad
  \begin{array}{@{}rcl@{}}
    p (0) & = & 0 \,,
    \\
    q (0) & = & 2 \,.
  \end{array}
\end{equation}
In the time interval $[0, \, 5]$, the speeds of $p$ and $q$ equal that
resulting from the LWR model at the initial
density~\eqref{eq:4}.\begin{figure}[!h] \centering
  \includegraphics[width=0.7\textwidth]{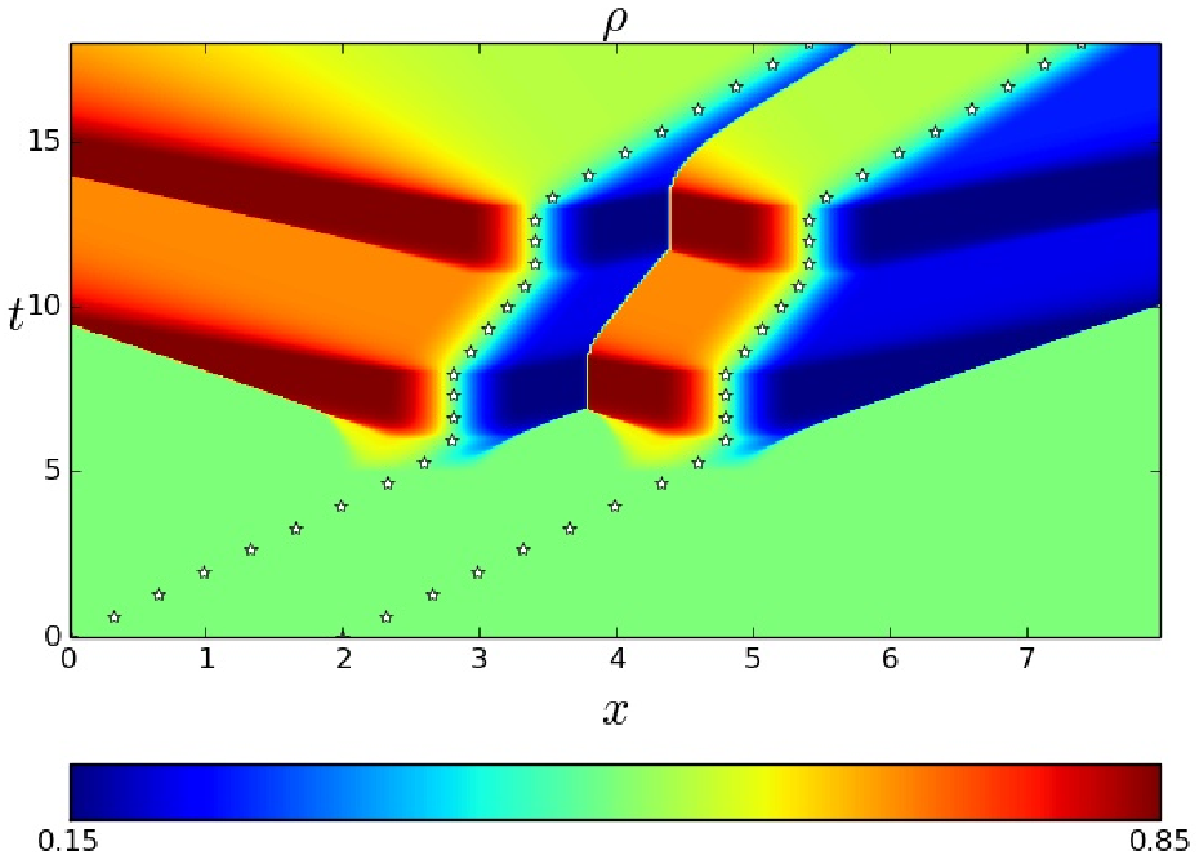}
  \Caption{Numerical integration of the
    model~\eqref{eq:CP}--\eqref{eq:CP1}--\eqref{eq:4}--(\ref{eq:pqt}). For
    $t \in [0, \, 5]$, the knowledge of the trajectories of $p$ and
    $q$ has no effects on the LWR description. For $t >5$, the effects
    of the jams revealed by the two cars are
    evident.\label{fig:questa}}
\end{figure}
Therefore, the two measuring cars have no effect whatsoever on the
evolution prescribed by the partial differential equation, see
Figure~\ref{fig:questa}.

The model allows to observe queues unpredictable for the LWR model,
thanks to the (supposedly) real time data provided by $p$
and $q$. Behind the measuring cars, the maximal density is reached,
while in front of them the road empties. Later, the jams
disappear. Between the two cars, the queue behind $q$ interacts with
the rarefaction formed in front of $p$.

\begin{figure}[!h] \centering
  \includegraphics[width=0.7\textwidth]{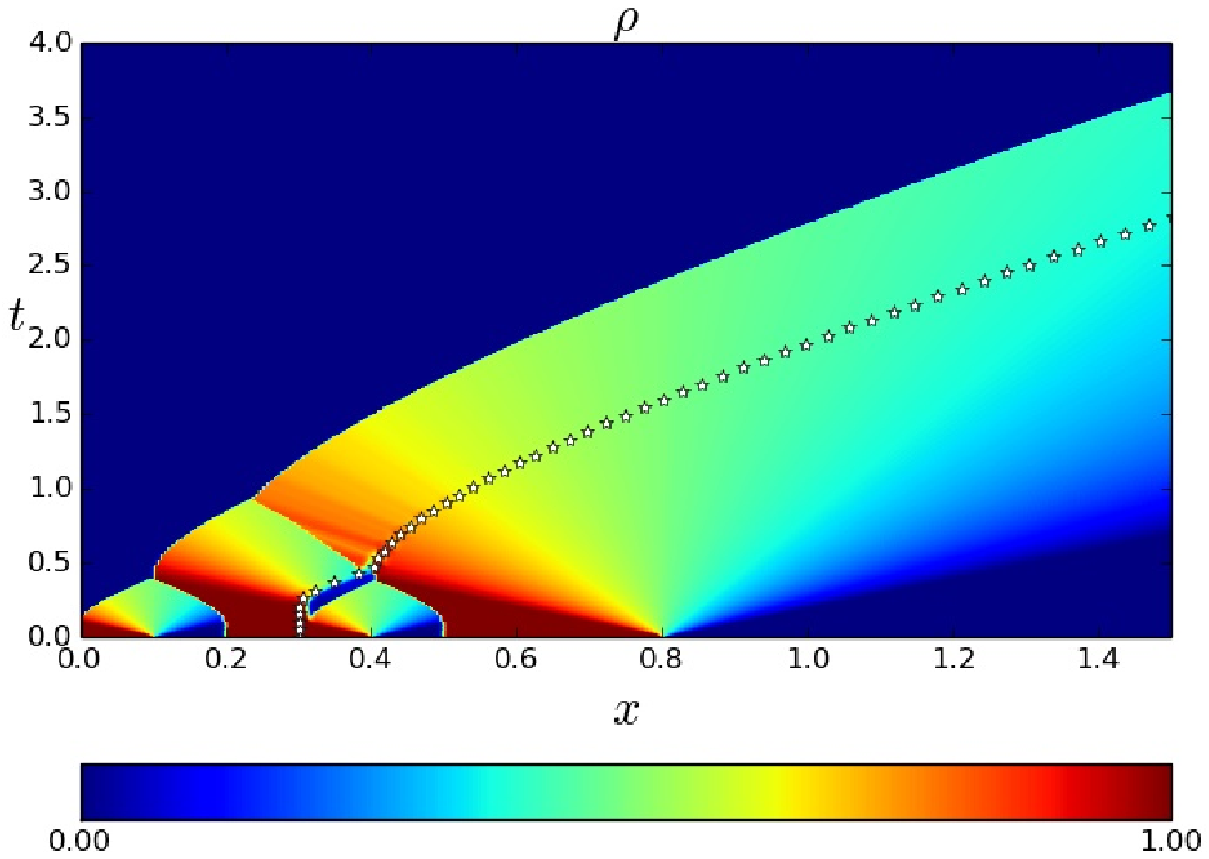}
  \Caption{Numerical integration
    of~\eqref{eq:CP}--\eqref{eq:CP1}--\eqref{eq:ID} with $\dot p =
    v\left(\rho (t, p (t)+)\right)$, resulting in the usual
    LWR evolution.\label{fig:int3}}
\end{figure}
As a second example, we choose the non constant initial datum
\begin{equation}
  \label{eq:ID}
  \rho_o(x)
  =
  \caratt{[0.001,0.1]}(x)
  +
  \caratt{[0.2,0.4]}(x)
  +
  \caratt{[0.5,0.8]}(x)
\end{equation}
for~\eqref{eq:CP} and assign to the unique measuring vehicle $p$ the
speed resulting from the speed law~\eqref{eq:5}, see
Figure~\ref{fig:int3}. The evolution prescribed by~\eqref{eq:3} is the
same as the one provided by the LWR model.

\begin{figure}[!h]
  \centering
  \includegraphics[width=0.7\textwidth]{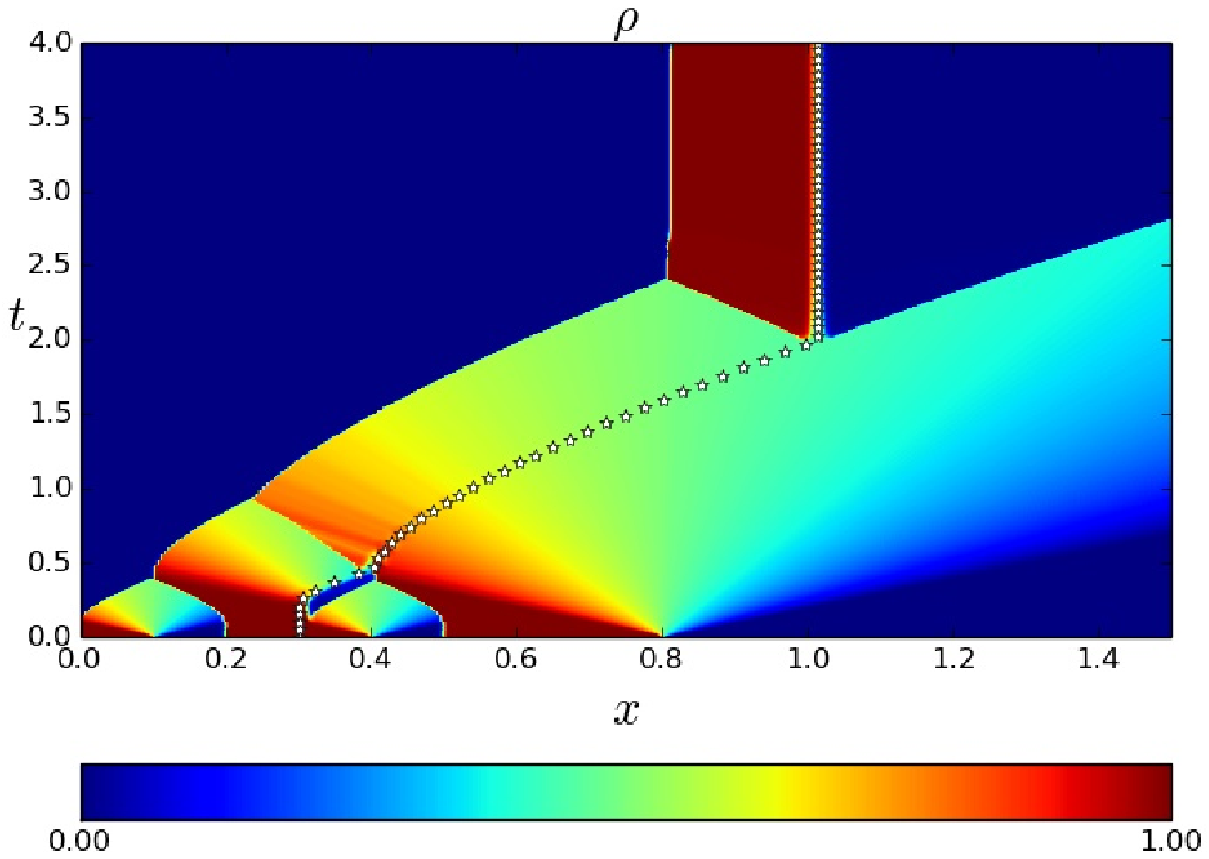}
  \Caption{Here, the experimental data are very different from the
    predictions of the LWR model. The measuring vehicle $p$ moves
    according to $\dot p = v\left(\rho (t, p (t)+)\right)$
    for $t \in [0, \, 2.0]$ and stops at $t=2.0$ due to, say, an
    accident. The model~\eqref{eq:CP}--\eqref{eq:CP1}--\eqref{eq:ID}
    is able to describe the resulting queue.\label{fig:int32}}
\end{figure}
Assume now that at time $t = 2.0$ the measuring car stops, due for
instance to some sort of accident. Then, the equation~\eqref{eq:3}
displays the formation of a standing queue, see
Figure~\ref{fig:int32}.

\begin{figure}[!h]
  \centering
  \includegraphics[width=0.7\textwidth]{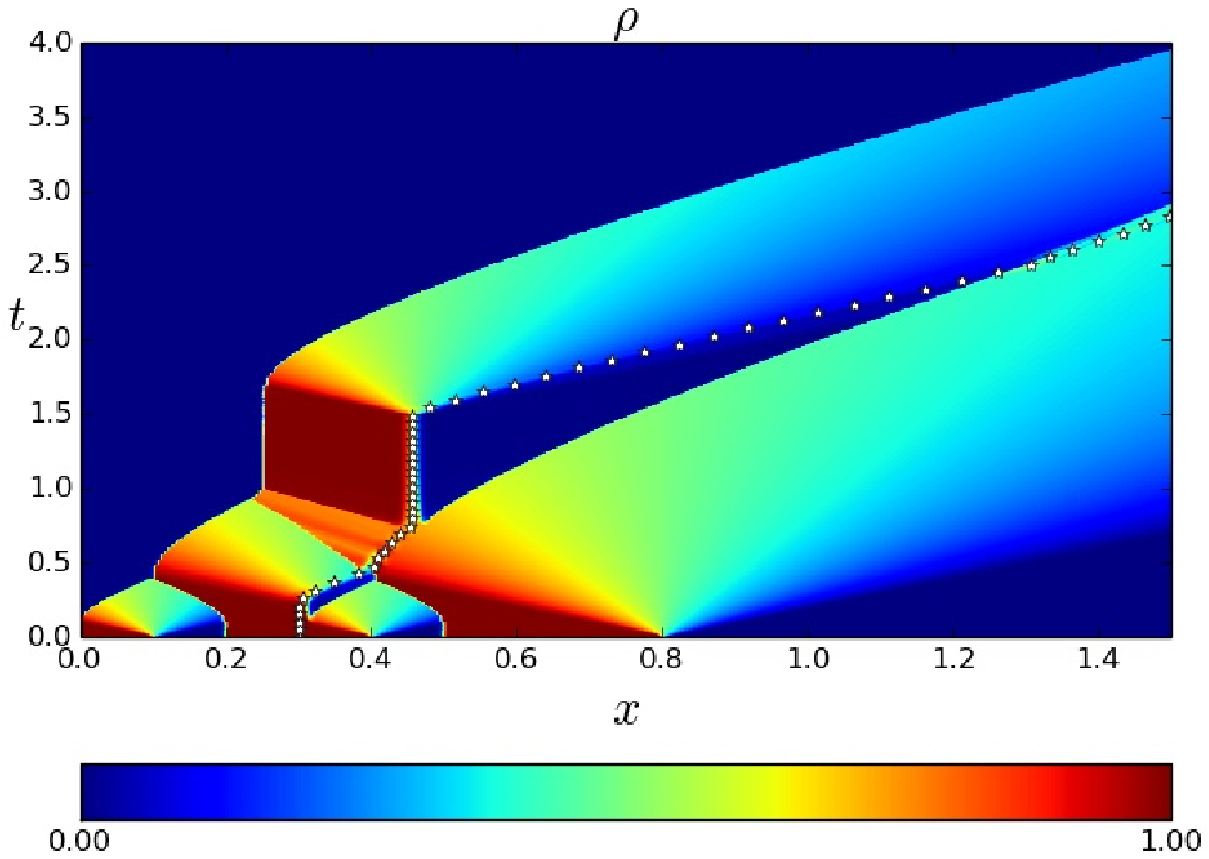}
  \Caption{The measuring vehicle $p$ moves according to $\dot p =
    v\left(\rho (t, p (t)+)\right)$ for $t \in [0, \,
    0.75]$ and stops for $t\in [0.75, \, 1.50]$.\label{fig:int33}}
\end{figure}

A slightly different situation is in Figure~\ref{fig:int33}. Here, the
measuring vehicle travels according to $\dot p = v\left(\rho \left(t,
    p (t)+\right)\right)$ for $t \in [0, \, 0.75]$. Then, the GPS data
show the presence of a standing queue at the location of $p$ during
the time interval $[0.75, \, 1.50]$. For $t > 1.50$, the measuring
vehicle travels again coherently with the prescription of the LWR
model: first at the maximal speed and then, after about $t \approx
2.50$, slowing down due to its reaching the vehicles in front.

The last two examples displayed in figures~\ref{fig:int32}
and~\ref{fig:int33}, when compared with Figure~\ref{fig:int3}, show
the dramatic changes in the description of traffic due to the
exploitation of real data. Anticipating the place and moment of an
accident is not possible, but within the framework
of~\eqref{eq:CP}--\eqref{eq:CP1} taking into account its consequences
is practically feasible.

\section{Technical Details}
\label{sec:TD}

\begin{proofof}{Lemma~\ref{lem:1}}
  We follow the main ideas of the proof
  of~\cite[Proposition~2.3]{Colomboguerra2010}.

  First, following~\cite[Chapter~6]{BressanLectureNotes} we use the
  wave front algorithm to obtain a sequence of piecewise constant
  approximate solutions to~\eqref{eq:lem1}.

  Fix $n \in \naturali \setminus\{0\}$ and call $f_n \in \C{0,1}
  (\reali; \reali)$, with $\Lip (f_n) \leq \Lip (f)$, the piecewise
  linear and continuous function such that $f_n (\rho) = f (\rho)$ for
  all $\rho \in 2^{-n} \interi$. Approximate the initial datum
  $\rho_o$ with a piecewise constant $\rho_o^n \in (\L1 \cap \BV)
  (\reali; \reali)$ such that $\rho_o^n (\reali) \subseteq
  2^{-n}\interi$ and $\tv (\rho_o^n) \leq \tv (\rho_o)$. Call $\rho^n$
  the exact solution to
  \begin{displaymath}
    \left\{
      \begin{array}{l}
        \partial_t \rho^n + \partial_x \left(f^n (\rho^n)\right) = 0
        \\
        \rho^n (0,x) = \rho_o^n (x)
      \end{array}
    \right.
  \end{displaymath}
  obtained gluing the solutions to the Riemann problems at the points
  of jump in $\rho_o^n$, see~\cite[Chapter~6]{BressanLectureNotes},
  Following~\cite[Lemma~4.4]{Colomboguerra2010}, locally the map $t
  \to \rho^n\left(t, \gamma_1 (t)\right)$ can be written as
  \begin{equation}
    \label{eq:lem1:rho}
    \rho^n \left(t, \gamma_1 (t)\right)
    =
    \sum_\alpha \rho_\alpha \chi_{[t_\alpha, t_{\alpha+1}[} (t)
    \quad \mbox{ with } \quad
    \gamma_1 (t_\alpha) = \lambda_\alpha \, t_\alpha + x_\alpha \,,
  \end{equation}
  where $t \to \lambda_\alpha \, t + x_\alpha$ supports a
  discontinuity in $\rho^n$ crossed by $\gamma_1$. Here, we intend
  that all states attained by $\rho^n$ in a neighborhood of
  $(t_\alpha,x_\alpha)$ appear in the sum \eqref{eq:lem1:rho},
  possibly with $t_{\alpha+1} = t_\alpha$.

  If $\eta \in \C1 (\reali; \reali)$ is such that $\norma{\eta}_{\C1}$
  is sufficiently small, there exists times $t'_\alpha$ such that
  \begin{equation}
    \label{eq:lem1:rhon}
    \rho^n \left(t, \gamma_1 (t)+\eta (t)\right)
    =
    \sum_\alpha \rho_\alpha \chi_{[t'_\alpha, t'_{\alpha+1}[} (t)
    \quad \mbox{ with } \quad
    \gamma_1 (t'_\alpha) = \lambda_\alpha \, t'_\alpha + x_\alpha \,.
  \end{equation}
  Hence
  \begin{eqnarray*}
    \left[\gamma_1 (t'_\alpha) + \eta (t'_\alpha)\right]
    -
    \gamma_1 (t_\alpha)
    & = &
    \lambda_\alpha (t'_\alpha - t_\alpha)
    \qquad \qquad \qquad \mbox{ on the other hand}
    \\
    \left[\gamma_1 (t'_\alpha) + \eta (t'_\alpha)\right]
    -
    \gamma_1 (t_\alpha)
    & = &
    \left[\gamma_1 (t'_\alpha) - \gamma_1 (t_\alpha)\right]
    +
    \eta (t'_\alpha)
    \\
    & = &
    \int_0^1 \dot
    \gamma_1 \left(\theta \, t'_\alpha + (1-\theta)t_\alpha\right) \d\tau
    \, (t'_\alpha - t_\alpha)
    + \eta (t'_\alpha)
  \end{eqnarray*}
  so that
  \begin{eqnarray*}
    \lambda_\alpha (t'_\alpha - t_\alpha)
    & = &
    \int_0^1 \dot
    \gamma_1 \left(\theta \, t'_\alpha + (1-\theta)t_\alpha\right) \d\tau
    \, (t'_\alpha - t_\alpha)
    + \eta (t'_\alpha)
    \\
    t'_\alpha - t_\alpha
    & = &
    \frac{\eta (t'_\alpha)}{\lambda_\alpha
      -
      \int_0^1 \dot
      \gamma_1 \left(\theta \, t'_\alpha + (1-\theta)t_\alpha\right) \d\tau}
    \\
    \modulo{t'_\alpha - t_\alpha}
    & = &
    \frac{\modulo{\eta (t'_\alpha)}}{\modulo{\lambda_\alpha
        -
        \int_0^1 \dot
        \gamma_1 \left(\theta \, t'_\alpha + (1-\theta)t_\alpha\right) \d\tau}}
    \\
    \modulo{t'_\alpha - t_\alpha}
    & \leq &
    \frac{1}{c} \, \norma{\eta}_{\C0} \,.
  \end{eqnarray*}
  Integrating the modulus of the difference between the
  terms~\eqref{eq:lem1:rho} and~\eqref{eq:lem1:rhon}, we obtain a
  first Lipschitz type estimate:
  \begin{eqnarray}
    \nonumber
    \int_0^T
    \modulo{\rho^n \left(t, \gamma_1 (t)\right)
      -
      \rho^n\left(t, \gamma_1 (t) + \eta (t)\right)} \d{t}
    & = &
    \sum_\alpha
    \modulo{\rho_\alpha - \rho_{\alpha-1}} (t_\alpha - t'_\alpha)
    \\
    \nonumber
    & \leq &
    \frac{1}{c} \, \norma{\eta}_{\C0}
    \sum_\alpha
    \modulo{\rho_\alpha - \rho_{\alpha-1}}
    \\
    \nonumber
    & \leq &
    \frac{1}{c} \, \norma{\eta}_{\C0}
    \tv \left\{\rho^n \left(\cdot, \gamma_1 (\cdot)\right), [0,T]\right\}
    \\
    \nonumber
    & \leq &
    \frac{1}{c} \, \norma{\eta}_{\C0}
    \tv (\rho_o^n)
    \\
    \label{eq:lem1:theta}
    & \leq &
    \frac{1}{c} \, \norma{\eta}_{\C0}
    \tv (\rho_o)
  \end{eqnarray}
  The proof is now completed as that
  of~\cite[Lemma~4.4]{Colomboguerra2010}. Introduce $\psi \colon [0,1]
  \to \reali$ by
  \begin{displaymath}
    \psi (\theta)    =
    \int_0^T
    \modulo{
      \rho^n
      \left(
        t, \theta \, \gamma_2 (t)
        +
        (1-\theta) \gamma_1 (t)
      \right)
    }
    \d{t}
  \end{displaymath}
  and observe that the above estimate~\eqref{eq:lem1:theta} ensures
  that $\psi$ is locally Lipschitz continuous and moreover
  $\modulo{\dot \psi} \leq \frac{1}{c} \, \tv (\rho_o) \,
  \norma{\gamma_2 - \gamma_1}_{\C0 ([0,T]; \reali)}$. Finally,
  \begin{eqnarray*}
    \int_0^T
    \modulo{
      \rho^n\left(t, \gamma_2 (t)\right) - \rho^n (t, \gamma_1 (t))} \d{t}
    & = &
    \psi (1) - \psi (0)
    \\
    & \leq &
    \norma{\dot \psi}_{\L\infty ([0,1]; \reali)}
    \\
    & \leq &
    \frac{1}{c} \,
    \tv (\rho_o) \,
    \norma{\gamma_2 - \gamma_1}_{\C0 ([0,T]; \reali)} \,.
  \end{eqnarray*}
  Thanks to the convergence of $\rho^n$ to $\rho$, an application of
  Lebesgue Dominated Convergence theorem completes the proof.
\end{proofof}

\begin{proofof}{Proposition~\ref{prop:positive}}
  Let $\rho_{\strut V}$ solve $\partial_t \rho + \partial_x \big(\rho
    \, v_{\strut V} (\rho)\big)=0$. Then, $\rho_{\strut V_2} (t,x) =
  \rho_{\strut V_1} \! \left(t, \frac{V_1}{V_2} \, x\right)$.  Indeed,
  \begin{displaymath}
    \partial_t \rho_{\strut V_2} (t,x)
    =
    \partial_t \rho_{\strut V_1} \left(t, \frac{V_1}{V_2} \, x\right)
    \quad \mbox{and} \quad
    \partial_x \rho_{\strut V_2} (t,x)
    =
    \frac{V_1}{V_2} \;
    \partial_x \rho_{\strut V_1} \left(t, \frac{V_1}{V_2} \, x\right)
  \end{displaymath}
  so that, setting $f_V (\rho) = \rho \, V (1-\rho)$, we have that
  $f_{V_2} = \frac{V_2}{V_1} \, f_{V_1}$ and
  \begin{eqnarray*}
    & &
    \partial_t \rho_{\strut V_2} (t,x)
    +
    \partial_x \left(f_{V_2}\left(\rho_{\strut V_2} (t,x)\right)\right)
    \\
    & = &
    \partial_t \rho_{\strut V_2} (t, x)
    +
    f'_{V_2} \left(\rho_{\strut V_2} (t,x)\right)
    \partial_x \rho_{\strut V_2} (t,x)
    \\
    & = &
    \partial_t \rho_{\strut V_1} \left(t, \frac{V_1}{V_2} \, x\right)
    +
    \frac{V_2}{V_1}
    f'_{V_1}\left(\rho_{\strut V_1} \left(t, \frac{V_1}{V_2} \, x\right) \right)
    \;
    \frac{V_1}{V_2} \;
    \partial_x \rho_{\strut V_1} \left(t, \frac{V_1}{V_2} \, x\right)
    \\
    & = &
    \partial_t \rho_{\strut V_1} \left(t, \frac{V_1}{V_2} \, x\right)
    +
    f'_{V_1}\left(\rho_{\strut V_1} \left(t, \frac{V_1}{V_2} \, x\right) \right)
    \;
    \partial_x \rho_{\strut V_1} \left(t, \frac{V_1}{V_2} \, x\right)
    \\
    & = &
    \partial_t \rho_{\strut V_1} \left(t,\frac{V_1}{V_2} \,x\right)
    +
    \partial_x \left(f_{V_1}\left(\rho_{\strut V_1} \left(t,\frac{V_1}{V_2} \,x\right) \right)\right)
    \\
    & = &
    0\, .
  \end{eqnarray*}
  We are lead to consider
  \begin{eqnarray*}
    \modulo{\mathcal{E} (V_2) - \mathcal{E} (V_1)}
    & = &
    \modulo{
      \int_0^T
      \modulo{
        \dot p (t)
        -
        v_{V_2}\left(\rho_{\strut V_2}\left(t, p (t)\right)\right)}
      \d{t}
      -
      \int_0^T
      \modulo{
        \dot p (t)
        -
        v_{V_1}\left(\rho_{\strut V_1}\left(t, p (t)\right)\right)}
      \d{t}
    }
    \\
    & \leq &
    \modulo{
      \int_0^T
      \modulo{
        v_{V_2}\left(\rho_{\strut V_2}\left(t, p (t)\right)\right)
        -
        v_{V_1}\left(\rho_{\strut V_1}\left(t, p (t)\right)\right)}
      \d{t}
    }
    \\
    & \leq &
    \modulo{
      \int_0^T
      \modulo{
        v_{V_2}\left(\rho_{\strut V_2}\left(t, p (t)\right)\right)
        -
        v_{V_2}\left(\rho_{\strut V_1}\left(t, p (t)\right)\right)}
      \d{t}
    }
    \\
    &  &
    +
    \modulo{
      \int_0^T
      \modulo{
        v_{V_2}\left(\rho_{\strut V_1}\left(t, p (t)\right)\right)
        -
        v_{V_1}\left(\rho_{\strut V_1}\left(t, p (t)\right)\right)}
      \d{t}
    }
    \\
    & \leq &
    \modulo{
      V_2
      \int_0^T
      \modulo{
        \rho_{\strut V_2}\left(t, p (t)\right)
        -
        \rho_{\strut V_1}\left(t, p (t)\right)}
      \d{t}
    }
    +
    \modulo{V_2 - V_1} t
    \\
    & \leq &
    V_2
    \modulo{
      \int_0^T
      \modulo{
        \rho_{\strut 1}\left(t, \frac{p (t)}{V_2}\right)
        -
        \rho_{\strut 1}\left(t, \frac{p (t)}{V_1}\right)}
      \d{t}
    }
    +
    \modulo{V_2 - V_1} t
  \end{eqnarray*}
  and to prove continuity we show that Lemma~\ref{lem:1} can be
  applied.  Indeed, by the maximum principle for conservation laws,
  $\rho (t,x) \geq \check \rho$ for a.e.~$(t,x) \in [0,T] \times
  \reali$. Moreover
  \begin{eqnarray*}
    f_1' \left(\rho\left(t, p (t)\right)\right)
    & = &
    1 - 2\rho \left(t, p (t)\right)
    \leq
    1 - 2 \infess_{x \in \reali} \rho_o
    \\
    \frac{\dot p (t)}{V_i}
    & \geq &
    \frac{\hat V}{V_i} (1-2\check\rho)\geq 1-2\check\rho
  \end{eqnarray*}
  Choosing now $c > 0$ such that $c < 2 (\infess_{x \in \reali} \rho_o
  - \check\rho)$, Lemma~\ref{lem:1} can be applied and we obtain
  \begin{eqnarray*}
    \modulo{
      \int_0^T
      \modulo{
        \rho_{\strut 1}\left(t, \frac{p (t)}{V_2}\right)
        -
        \rho_{\strut 1}\left(t, \frac{p (t)}{V_1}\right)}
      \d{t}
    }
    & \leq &
    \frac{1}{c} \,
    \tv(\rho_o) \,
    \norma{\frac{p (t)}{V_2} - \frac{p (t)}{V_1}}_{\C0 ([0,T]; \reali)}
    \\
    & \leq &
    \frac{1}{c \, \check V} \,
    \tv(\rho_o) \,
    \norma{p}_{\C0 ([0,T]; \reali)} \,
    \modulo{V_2 - V_1}  \,,
  \end{eqnarray*}
  completing the proof.
\end{proofof}

\begin{lemma}
  \label{lem:conca}
  Let $v \in \C2 ([0,1]; [0,V])$ be such that $v' \leq 0$ and $\rho
  \to \rho \, v (\rho)$ is strictly concave. Then, the map $\rho \to
  \, \frac{\rho \, w \, v (\rho)}{w + v (\rho)}$ is strictly concave
  for all $w > 0$.
\end{lemma}

\begin{proof}
  Call $q (\rho) = \rho \, v (\rho)$ and $q_w (\rho) = \frac{\rho
    \, w \, v (\rho)}{w + v (\rho)}$. By direct computations,
  \begin{eqnarray*}
    q_w' (\rho)
    & = &
    w \, \frac{v^2 (\rho) + w \, v (\rho) + w \, \rho \, v' (\rho)}{\left(w + v (\rho)\right)^2}
    \\
    q_w'' (\rho)
    & = &
     w^2 \, \frac{\left(v (\rho) + w\right) \, q'' (\rho) - 2 \, \rho \, \left(v' (\rho)\right)^2}{\left(v (\rho) + w\right)^3}
  \end{eqnarray*}
  which clearly shows that $q_w'' (\rho) \leq 0$, since $q'' \leq 0$.
\end{proof}

The next regularity result is of use below.

\begin{lemma}
  \label{lem:reg}
  Let~\textbf{(v)}, \textbf{(p)} and~\textbf{($\boldsymbol{\chi}$)}
  hold. Then,
  \begin{enumerate}[(i)]
  \item the function $\mathcal{V}$ defined in~\eqref{eq:CP1} is
    continuous;
  \item the map $x \to \mathcal{V} (t,x,\rho)$ is uniformly Lipschitz
    continuous for $t \in \reali^+$ and $\rho \in [0,1]$;
  \item the map $\rho \to \mathcal{V} (t,x,\rho)$ is uniformly
    Lipschitz continuous for $t \in \reali^+$ and $x \in \reali$;
  \item the map $\rho \to \partial_x \mathcal{V} (t,x,\rho)$ is
    uniformly Lipschitz continuous for $t \in \reali^+$ and $x \in
    \reali$.
  \end{enumerate}
\end{lemma}

\begin{proof}
  Call $P = \Lip (p)$ and observe that the map
  \begin{displaymath}
    \begin{array}{ccc}
      [0,P] \times [0,V] & \to &\reali
      \\
      (\dot p, v) & \to &
      \left\{
        \begin{array}{ll}
          \displaystyle
          \frac{\dot p \, v}{\dot p + v} & (\dot p, v) \neq (0,0)
          \\
          0 & (\dot p, v) = (0,0)
        \end{array}
      \right.
    \end{array}
  \end{displaymath}
  is continuous and non negative. The continuity of $\mathcal{V}$
  immediately follows, proving~(i). Call $M = \max_{[0,P] \times
    [0,V]} \frac{\dot p \, v}{\dot p + v}$. Then,
  \begin{eqnarray*}
    \modulo{\mathcal{V} (t,x_2,\rho) - \mathcal{V} (t,x_1,\rho)}
    & \leq &
    (M + V) \,
    \modulo{
      \chi\left(x_2 - \dot p (t)\right)
      -
      \chi\left(x_1 - \dot p (t)\right)}
    \\
    & \leq &
    (M+V) \, \Lip (\chi) \, \modulo{x_2-x_1} \,,
  \end{eqnarray*}
  completing the proof of~(ii). Direct computations show that a
  Lipschitz constant for the map $\rho \to \mathcal{V} (t,x,\rho)$ is
  $2\Lip (v)$, proving~(iii). Finally,entirely analogous computations
  ensure that that $\rho \to \partial_x \mathcal{V} (t,x,\rho)$ is
  Lipschitz continuous with Lipschitz constant $\left(1+\Lip
    (\chi)\right) \Lip (v)$.
\end{proof}

\begin{lemma}
  \label{lem:speriamo}
  Let~\textbf{(v)} hold and fix a positive $P$. Consider the map
  \begin{equation}
    \label{eq:g}
    \begin{array}{ccccccc}
      g & \colon & [0,1] & \times & [0,P] & \to & \reali
      \\
      & & (\rho &, & q) & \to &
      \displaystyle
      \frac{q\, \rho \, v (\rho)}{q + v (\rho)} \,.
    \end{array}
  \end{equation}
  Then, there exists a $L>0$ such that for all $(\rho_1,q_1),
  (\rho_2,q_2) \in [0,1] \times [0,P]$,
  \begin{displaymath}
    \modulo{g (\rho_1,q_1) - g (\rho_1,q_2) - g (\rho_2,q_1) + g (\rho_2,q_2)}
    \leq
    L \;
    \modulo{\rho_1 - \rho_2} \;
    \modulo{q_1 - q_2} \,.
  \end{displaymath}
\end{lemma}

\begin{proof}
  Compute first the partial derivative
  \begin{displaymath}
    \partial_q g (\rho,q)
    =
    \frac{\rho \, v^2 (\rho)}{\left(q + v (\rho)\right)^2} \,.
  \end{displaymath}
  By~\textbf{(v)}, the map $\rho \to \partial_q g (\rho,q)$ is
  Lipschitz continuous on $[0,1] \times [0,P]$, hence it is
  a.e.~differentiable with respect to $\rho$. Moreover
  \begin{displaymath}
    \partial^2_{\rho q} g (\rho,q)
    =
    v (\rho)\left(v (\rho) + 2\rho\, v' (\rho)\right)
    -
    2 \frac{\rho\, v^2 (\rho) \, v' (\rho)}{q+v (\rho)}
  \end{displaymath}
  and, clearly, $\sup_{[0,1]\times[0,P]} \modulo{\partial^2_{\rho q} g
    (\rho,q)} < +\infty$.  We can then write:
  \begin{eqnarray*}
    & &
    \modulo{g (\rho_1,q_1) - g (\rho_1,q_2) - g (\rho_2,q_1) + g (\rho_2,q_2)}
    \\
    & = &
    \modulo{
      \int_0^1
      \partial_q g\left(\rho_1, (1-\theta)q_1 + \theta q_2\right) \d\theta
      (q_1-q_2)
      -
      \int_0^1
      \partial_q g\left(\rho_2, (1-\theta)q_1 + \theta q_2\right) \d\theta
      (q_1-q_2)
    }
    \\
    & = &
    \modulo{
      \int_0^1 \! \int_0^1
      \partial^2_{\rho q} g
      \left(
        (1-\eta)\rho_1 + \eta \rho_2, (1-\theta)q_1 + \theta q_2
      \right)
      \d\eta \d\theta
    }
    \modulo{\rho_1-\rho_2} \,\modulo{q_1 - q_2}
    \\
    & \leq &
    \left(\sup_{[0,1]\times[0,P]} \modulo{\partial^2_{\rho q} g (\rho,q)}\right)
    \modulo{\rho_1-\rho_2} \,\modulo{q_1 - q_2} \,,
  \end{eqnarray*}
  completing the proof.
\end{proof}

\begin{proofof}{Proposition~\ref{prop:exi}}
  In the present setting, the assumptions
  of~\cite[Theorem~2]{Panov_2010_1} read:
  \begin{enumerate}[(1)]
  \item $f$ is a Caratheodory vector field on $\reali^+ \times
    \reali$, i.e., for a.e.~$(t,x) \in \reali^+ \times \reali$, the
    map $\rho \to f (t,x,\rho)$ is continuous and for all $\rho \in
    [0,1]$ the map $(t,x) \to f (t,x,u)$ is measurable.
  \item The map $ \rho \to f (t,x,\rho)$ is non--degenerate, i.e., it
    is not affine on non trivial intervals.
  \item For some $a,b \in [0,1]$, for all $(t,x) \in \reali^+ \times
    \reali$, $f (t,x,a) = f (t,x,b) = 0$ and the map $(t,x) \to
    \max_{\rho \in [a,b]} \modulo{f (t,x,\rho)}$ is in $\Lloc{q}
    (\reali^+ \times \reali; \reali)$ for a $q > 2$.
  \end{enumerate}
  \noindent Note that~(1) directly follows from~\textbf{(v)},
  \textbf{(p)} and~\textbf{($\boldsymbol{\chi}$)}. The requirement~(2)
  follows from~\textbf{(v)}, \textbf{($\boldsymbol\chi$)}
  and~\textbf{(p)}, indeed they ensure that $\mathcal{V}$ is a convex
  combination of strictly concave functions. Condition~(3) can be
  easily verified, with $a=0$ and $b=1$, thanks to~\textbf{(v)} and
  since $f (t,x,\rho) \in [0, \max\{\Lip (\dot p), V\}]$. Hence,
  \cite[Theorem~2]{Panov_2010_1} applies.
\end{proofof}

\begin{proofof}{Proposition~\ref{prop:def}}
  Observe that by (ii)~of Lemma~\ref{lem:reg}, the distributional
  derivative of the map $(t,x) \to \partial_x f (t,x,k)$ has no
  singular part. Hence, $\mu^ks$ vanishes
  in~\eqref{eq:DefPanov}. Moreover, choosing a test function with
  support in $\pint{\reali}^+ \times \reali$ makes the latter summand
  in the left hand side of~\eqref{eq:DefPanov} vanish. Hence,
  \eqref{eq:DefPanov} implies~\eqref{eq:DefKruzkov}. Finally, the
  condition~\eqref{eq:DefKruzkov0} on the initial datum is known to be
  implied by the stronger~\eqref{eq:DefPanov}, see for
  instance~\cite[Formula~(10)]{Panov_2010_1}.
\end{proofof}

\begin{lemma}
  \label{lem:Fregatura}
  \cite[Theorem~1.1]{KarlsenRisebro2003} does not apply
  to~\eqref{eq:CP1}, since the one sided Lipschitz
  condition~\cite[Formula~(1.7)]{KarlsenRisebro2003} may fail to hold.
\end{lemma}

\begin{proof}
  In the present setting, due to the absence of the parabolic term and
  of the source on the right hand side of~\eqref{eq:3}, the
  assumptions in~\cite{KarlsenRisebro2003} necessary to
  apply~\cite[Theorem~1.1]{KarlsenRisebro2003} on the time interval
  $[0,T]$, for any $T>0$, are the following:
  \begin{enumerate}[(1)]
  \item $f (t,x,0) = \partial_x f (t,x,0) = 0$ for a.e.~$t \in
    \reali^+$ and for all $x \in \reali$.
  \item The map $(t,x) \to f (t,x,\rho)$ is in $\L1 (\reali^+;
    \Wloc{1,1} (\reali;\reali))$ for all $\rho \in [0,1]$.
  \item For any $T>0$, the map $(t,x) \to \partial_x f (t,x,\rho)$ is
    in $\L1 ([0,T]; \L\infty (\reali;\reali))$ for all $\rho \in
    [0,1]$.
  \item There exists a positive $L$ such that for a.e.~$t \in
    \reali^+$, for all $x \in \reali$ and all $\rho_1,\rho_2 \in
    [0,1]$
    \begin{displaymath}
      \modulo{f (t,x,\rho_2) - f (t,x,\rho_1)}
      +
      \modulo{\partial_x f (t,x,\rho_2) - \partial_xf (t,x,\rho_1)}
      \leq
      L \, \modulo{\rho_2 - \rho_1} \,.
    \end{displaymath}
  \item Define $F (t,x,\rho_1,\rho_2) = \sgn (\rho_1 - \rho_2) \,
    \left(f (t,x,\rho_1) - f (t,x,\rho_2)\right)$. There exists a
    positive $C$ such that for a.e.~$t_1,t_2 \in \reali^+$, for all
    $x_1,x_2 \in \reali$ and all $\rho_1,\rho_2 \in [0,1]$
    \begin{displaymath}
      \left(
        F (t_1,x_1,\rho_1,\rho_2) -
        F (t_2,x_2,\rho_1,\rho_2)
      \right)
      (x_1 - x_2)
      \geq - C \, \modulo{\rho_1 - \rho_2} \, (x_1-x_2)^2 \,.
    \end{displaymath}
  \end{enumerate}
  \noindent We first prove that~(1), (2), (3) and~(4) hold.

  Note that~(1) is immediate, by~\eqref{eq:CP1}
  and~\textbf{(v)}. By~(2) we mean that for any compact set $K \subset
  \reali$, for any positive $T$ and for any $\rho \in [0,1]$, the map
  $t \to \int_K \left( \modulo{f (t,x,\rho)} + \modulo{\partial_x f
      (t,x,\rho)} \right) \d{x}$ is in $\L1 (0,T];\reali)$, which is
  immediate since both $f$ and $\partial_x f$ are uniformly bounded,
  thanks to~\eqref{eq:CP1} and Lemma~\ref{lem:reg}. This uniform bound
  on $\partial_x f$ also proves~(3). At~(4), the Lipschitz continuity
  of $\rho \to f (t,x,\rho)$, respectively $\rho \to \partial_x f
  (t,x,\rho)$, is proved in~(iii), respectively~(iv), of
  Lemma~\ref{lem:reg}.

  Finally, we note that~(5) may fail to hold, due to the dependence of
  the left hand side on time. Assume, for instance, that
  \begin{displaymath}
    \begin{array}{rcl@{\qquad}rcl@{\qquad}rcl@{\qquad}rcl}
      v (\rho) & = & 1-\rho
      &
      t_1 & = & 0
      &
      x_1 & =  & 0
      &
      \rho_1 & = & 1/2
      \\
      p (t) & = & t
      &
      t_2 & = & 2
      &
      x_2 & = & \epsilon
      &
      \rho_2 & = & 0
    \end{array}
  \end{displaymath}
  the, condition~(5) amounts to require the existence of a constant
  $C$ such that $\epsilon/6 \leq C \, \epsilon^2$, which is not
  possible.
\end{proof}

\begin{proofof}{Proposition~\ref{prop:Uni}}
  We exploit the doubling of variables method, see~\cite{Kruzkov}. To
  this aim, assume that $\rho_1$ and $\rho_2$ are two solutions
  to~\eqref{eq:CP} in the sense of Definition~\ref{def:sol}. Let $\psi
  = \psi (t,x,s,y)$ be in $\Cc\infty\left((\pint{\reali}^+ \times
    \reali)^2; \reali^+\right)$, write~\eqref{eq:DefKruzkov} for $\rho
  = \rho_1 (t,x)$ and for $k = \rho_2 (s,y)$, integrate the resulting
  inequality on $\reali^+ \times \reali$, to obtain:
  \begin{displaymath}
    \begin{array}{@{}l@{}l@{}l@{}}
      \displaystyle
      \int_{\reali^+} \!\! \int_{\reali} \! \int_{\reali^+} \!\!  \int_{\reali}
      &
      \displaystyle
      \Big[
      \modulo{\rho_1 (t,x) - \rho_2(s,y)} \, \partial_t \psi(t,x,s,y)
      \\
      &
      \displaystyle
      +
      \sgn\left(\rho_1 (t,x) - \rho_2(s,y)\right)
      \left(
        f\left(t, x, \rho_1 (t,x)\right)
        -
        f \left(t, x, \rho_2(s,y)\right)
      \right)
      \partial_x \psi(t,x,s,y)
      \\
      &
      \displaystyle
      -
      \sgn \left(\rho_1 (t,x) - \rho_2(s,y)\right)
      \partial_x f \left(t,x,\rho_2(s,y)\right) \, \psi(t,x,s,y)
      \Big] \d{x} \d{t} \d{y} \d{s}
      &
      \geq 0.
    \end{array}
  \end{displaymath}
  Repeat now the same procedure exchanging the roles of $\rho_1$ and
  $\rho_2$, obtaining
  \begin{displaymath}
    \begin{array}{@{}l@{}l@{}l@{}}
      \displaystyle
      \int_{\reali^+} \!\! \int_{\reali} \! \int_{\reali^+} \!\!  \int_{\reali}
      &
      \displaystyle
      \Big[
      \modulo{\rho_1 (t,x) - \rho_2(s,y)} \, \partial_s \psi(t,x,s,y)
      \\
      &
      \displaystyle
      +
      \sgn\left(\rho_2 (s,y) - \rho_1(t,x)\right)
      \left(
        f\left(s, y, \rho_2 (s,y)\right)
        -
        f \left(s, y, \rho_1(t,x)\right)
      \right)
      \partial_y \psi(t,x,s,y)
      \\
      &
      \displaystyle
      -
      \sgn \left(\rho_2 (s,y) - \rho_1(t,x)\right)
      \partial_y f \left(s,y,\rho_1(t,x)\right) \, \psi(t,x,s,y)
      \Big] \d{x} \d{t} \d{y} \d{s}
      &
      \geq 0.
    \end{array}
  \end{displaymath}
  The sum of the latter two inequalities above yields
  \begin{displaymath}
    \begin{array}{@{}l@{}}
      \displaystyle
      \int_{\reali^+} \!\! \int_{\reali} \! \int_{\reali^+} \!\!  \int_{\reali}
      \Big[
      \modulo{\rho_1 (t,x) - \rho_2(s,y)}
      \left(\partial_t \psi(t,x,s,y) + \partial_s \psi(t,x,s,y)\right)
      \\
      +
      \sgn\left(\rho_1 (t,x) - \rho_2(s,y)\right)
      \left(
        f\left(t, x, \rho_1 (t,x)\right)
        -
        f \left(t, x, \rho_2(s,y)\right)
      \right)
      \partial_x \psi(t,x,s,y)
      \\
      +
      \sgn\left(\rho_1 (t,x) - \rho_2(s,y)\right)
      \left(
        f\left(s, y, \rho_1 (t,x)\right)
        -
        f \left(s, y, \rho_2(s,y)\right)
      \right)
      \partial_y \psi(t,x,s,y)
      \\
      +
      \sgn \left(\rho_1 (t,x) - \rho_2(s,y)\right) \!
      \left(
        \partial_y f \!\left(s,y,\rho_1(t,x)\right)
        -
        \partial_x f \! \left(t,x,\rho_2(s,y)\right)
      \right) \!
      \psi(t,x,s,y)
      \Big] \!\d{x} \d{t} \d{y} \d{s}
      \geq 0.
    \end{array}
  \end{displaymath}
  Following~\cite{KarlsenRisebro2003}, since
  \begin{displaymath}
    \begin{array}{l@{}l@{}l}
      &
      \partial_x
      \Big[    \left(
        f\left(s, y, \rho_2 (s,y)\right)
        -
        f \left(t, x, \rho_2(t,x)\right)
      \right)
      \psi (t,x,s,y)
      \Big]\\
      &
      =
      -\partial_xf \left(t, x, \rho_2(t,x)\right)+f\left(s, y, \rho_2 (s,y)\right) \partial_x \psi (t,x,s,y)-f \left(t, x, \rho_2(t,x)\right) \partial_x \psi (t,x,s,y)\,,
    \end{array}
  \end{displaymath}
  we get

  \begin{displaymath}
    \begin{array}{l@{}l@{}l}
      &
      \Big[f \left(t, x, \rho_1(t,x)\right)-f\left(t, x, \rho_2 (s,y)\right)\Big] \partial_x \psi (t,x,s,y)- \partial_x f\left(t, x, \rho_2 (s,y)\right)\psi (t,x,s,y)\\
      &
      =
      \Big[f \left(t, x, \rho_1(t,x)\right)-f\left(s, y, \rho_2 (s,y)\right)\Big] \partial_x \psi (t,x,s,y)\\
      &
      +
      \partial_x
      \Big[    \left(
        f\left(s, y, \rho_2 (s,y)\right)
        -
        f \left(t, x, \rho_2(t,x)\right)
      \right)
      \psi (t,x,s,y)
      \Big]
    \end{array}
  \end{displaymath}

  and similarly

  \begin{displaymath}
    \begin{array}{l@{}l@{}l}
      &
      \Big[f \left(s, y, \rho_1(t,x)\right)-f\left(s, y, \rho_2 (s,y)\right)\Big] \partial_y \psi (t,x,s,y)- \partial_y f\left(s, y, \rho_1 (t,x)\right)\psi (t,x,s,y)\\
      &
      =
      \Big[f \left(t, x, \rho_1(t,x)\right)-f\left(s, y, \rho_2 (s,y)\right)\Big] \partial_y \psi (t,x,s,y)\\
      &
      -
      \partial_y
      \Big[    \left(
        f\left(t, x, \rho_1 (t,x)\right)
        -
        f \left(s, y, \rho_1(t,x)\right)
      \right)
      \psi (t,x,s,y)
      \Big]\,.
    \end{array}
  \end{displaymath}
  Thus, following where possible the notation
  in~\cite[p.~1093]{KarlsenRisebro2003}, we have
  \begin{equation}
    \label{eq:3.35}
    \int_{\reali^+} \!\! \int_{\reali} \! \int_{\reali^+} \!\!  \int_{\reali}
    \left(
      I_0 + I_1+ I_3
    \right)
    \d{x} \d{t} \d{y} \d{s}
    \geq 0
  \end{equation}
  where
  \begin{eqnarray*}
    I_0
    & = &
    \modulo{\rho_1 (t,x) - \rho_2(s,y)}
    \left(\partial_t \psi(t,x,s,y) + \partial_s \psi(t,x,s,y)\right)
    \\
    I_1
    & = &
    \sgn\left(\rho_1 (t,x) - \rho_2(s,y)\right)
    \Big[f \left(t, x, \rho_1(t,x)\right)-f\left(s, y, \rho_2 (s,y)\right)\Big]
\\
& &
\qquad \times
\left( \partial_x \psi (t,x,s,y)+\partial_y \psi (t,x,s,y)\right)
    \\
    I_3
    & = &
    \sgn\left(\rho_1 (t,x) - \rho_2(s,y)\right)
    \Big[
    \partial_x
    \left(    \left(
        f\left(s, y, \rho_2 (s,y)\right)
        -
        f \left(t, x, \rho_2(s,y)\right)
      \right)
      \psi (t,x,s,y)
    \right)
    \\
    & &
    \qquad\qquad-
    \partial_y
    \left(    \left(
        f\left(t, x, \rho_1 (t,x)\right)
        -
        f \left(s, y, \rho_1(t,x)\right)
      \right)
      \psi (t,x,s,y)
    \right)
    \Big]
    \,.
  \end{eqnarray*}
  We proceed towards the choice of the test functions introducing
  first a map
  \begin{displaymath}
    \delta \in \Cc\infty (\reali;\reali^+)
    \quad \mbox{ such that } \quad
    \spt \delta \subseteq [-1,1]
    \,,\quad \delta (-\xi) = \delta (\xi) \,,\quad
    \quad \mbox{ and } \quad
    \int_{\reali} \delta (\xi) \d{\xi} = 1 \,.
  \end{displaymath}
  Moreover, for $r>0$, let
  \begin{equation}
    \label{eq:deltar}
    \delta_r (t) = \delta (t/r)/r
    \quad \mbox{ and } \quad
    \omega_r (x) =  \delta\left(x^2 / r^2\right) / (2r) \,.
  \end{equation}
  and for a $\phi \in \Cc\infty (\reali^+ \times \reali; \reali^+)$,
  choose
  \begin{displaymath}
    \psi (t,x,s,y)
    =
    \phi\left(\frac{t+s}{2}, \frac{x+y}{2}\right)
    \;
    \omega_r\left(\frac{x-y}{2}\right)
    \;
    \delta_r\left(\frac{t-s}{2}\right) \,.
  \end{displaymath}
  We then rewrite~\eqref{eq:3.35} as
  \begin{equation}
    \label{eq:3.38}
    \int_{\reali^+} \! \int_{\reali} \! \int_{\reali^+} \! \int_{\reali}
    \left[
      \left(
        \bar I_0 + \bar I_1 + \bar I_3
      \right)
      \omega_r\left(\frac{x-y}{2}\right)
      \delta_r\left(\frac{t-s}{2}\right)
      +
      \bar I_5  \partial_x \omega_r\left(\frac{x-y}{2}\right)
    \right]
    \d{x} \d{t} \d{y} \d{s}
    \geq 0
  \end{equation}
  where
  \begin{eqnarray*}
    \bar I_0
    & = &
    \modulo{\rho_1 (t,x) - \rho_2 (s,y)}
    \left(
      \partial_t \phi\left(\frac{t+s}{2}, \frac{x+y}{2}\right)
      +
      \partial_s \phi\left(\frac{t+s}{2}, \frac{x+y}{2}\right)
    \right)
    \\
    \bar I_1
    & = &
    \sgn\left(\rho_1 (t,x) - \rho_2 (s,y)\right)
    \left(
      f \left(t,x,\rho_1 (t,x)\right)
      -
      f\left(s,y,\rho_2 (s,y)\right)
    \right)
    \\
    & &
    \quad
    \times
    \left(
      \partial_x \phi\left(\frac{t+s}{2}, \frac{x+y}{2}\right)
      +
      \partial_y \phi\left(\frac{t+s}{2}, \frac{x+y}{2}\right)
    \right)
    \\
    \bar I_3
    & = &
    \sgn\left(\rho_1 (t,x) - \rho_2 (s,y)\right)
    \\
    & &
    \! \times \!
    \Big[ \!
    \left(
      \partial_x \! \left[
        f \left(s,y,\rho_2 (s,y)\right)
        -
        f\left(t,x,\rho_2 (s,y)\right)
      \right]
      -
      \partial_y \! \left[
        f \left(t,x,\rho_1 (t,x)\right)
        -
        f\left(s,y,\rho_1 (t,x)\right)
      \right]
    \right)
    \\
    & &
    \qquad\qquad
    \times
    \phi\left(\frac{t+s}{2},\frac{x+y}{2}\right)
    \\
    & &
    \qquad\quad
    +
    \left(
      f \left(s,y,\rho_2 (s,y)\right)
      -
      f\left(t,x,\rho_2 (s,y)\right)
    \right)
    \partial_x \phi\left(\frac{t+s}{2},\frac{x+y}{2}\right)
    \\
    & &
    \qquad\quad
    +
    \left(
      f \left(t,x,\rho_1 (t,x)\right)
      -
      f\left(s,y,\rho_1 (t,x)\right)
    \right)
    \partial_y \phi\left(\frac{t+s}{2},\frac{x+y}{2}\right)
    \Big]
    \\
    \bar I_5
    & = &
    \sgn\left(\rho_1 (t,x) - \rho_2 (s,y)\right)
    \\
    & &
    \times
    \left(
      f\left(t,x,\rho_1 (t,x)\right)
      -
      f\left(t,x,\rho_2 (s,y)\right)
      -
      f\left(s,y,\rho_1 (t,x)\right)
      +
      f\left(s,y,\rho_2 (s,y)\right)
    \right)
    \\
    & &
    \times
    \phi\left(\frac{t+s}{2}, \frac{x+y}{2}\right) \;
    \delta_r\left(\frac{t-s}{2}\right) \,.
  \end{eqnarray*}
  We first estimate the term $\bar I_5$ as follows:
  \begin{eqnarray*}
    \bar I_5
    & = &
    \sgn\left(\rho_1 (t,x) - \rho_2 (s,y)\right)
    \\
    & &
    \times
    \left(
      f\left(t,x,\rho_1 (t,x)\right)
      -
      f\left(t,x,\rho_2 (s,y)\right)
      -
      f\left(s,y,\rho_1 (t,x)\right)
      +
      f\left(s,y,\rho_2 (s,y)\right)
    \right)
    \\
    & &
    \times
    \phi\left(\frac{t+s}{2}, \frac{x+y}{2}\right) \;
    \delta_r\left(\frac{t-s}{2}\right)
  \end{eqnarray*}
  We now estimate the term in parentheses above:
  \begin{eqnarray*}
    & &
    f\left(t,x,\rho_1 (t,x)\right)
    -
    f\left(t,x,\rho_2 (s,y)\right)
    -
    f\left(s,y,\rho_1 (t,x)\right)
    +
    f\left(s,y,\rho_2 (s,y)\right)
    \\
    & = &
    \chi\left(x-p (t)\right)
    \left(
      \frac{\dot p (t) \, \rho_1 \, v (\rho_1)}{\dot p (t) + v (\rho_1)}
      -
      \frac{\dot p (t) \, \rho_2 \, v (\rho_2)}{\dot p (t) + v (\rho_2)}
    \right)
    \\
    & &
    \qquad -
    \chi\left(y-p (s)\right)
    \left(
      \frac{\dot p (s) \, \rho_1 \, v (\rho_1)}{\dot p (s) + v (\rho_1)}
      -
      \frac{\dot p (s) \, \rho_2 \, v (\rho_2)}{\dot p (s) + v (\rho_2)}
    \right)
    \\
    & &
    +
    \left(1-\chi\left(x-p (t)\right)\right)
    \left(\rho_1 \, v (\rho_1) - \rho_2 \, v (\rho_2)\right)
    \\
    & &
    \qquad
    -
    \left(1-\chi\left(x-p (s)\right)\right)
    \left(\rho_1 \, v (\rho_1) - \rho_2 \, v (\rho_2)\right)
    \\
    & = &
    \left(
      \chi\left(x-p (t)\right) - \chi\left(y-p (s)\right)
    \right)
    \left(
      \frac{\dot p (t) \, \rho_1 \, v (\rho_1)}{\dot p (t) + v (\rho_1)}
      -
      \frac{\dot p (t) \, \rho_2 \, v (\rho_2)}{\dot p (t) + v (\rho_2)}
    \right)
    \\
    & &
    +
    \chi\left(y-p (s)\right)
    \left(
      \frac{\dot p (t) \, \rho_1 \, v (\rho_1)}{\dot p (t) + v (\rho_1)}
      -
      \frac{\dot p (t) \, \rho_2 \, v (\rho_2)}{\dot p (t) + v (\rho_2)}
      -
      \frac{\dot p (s) \, \rho_1 \, v (\rho_1)}{\dot p (s) + v (\rho_1)}
      +
      \frac{\dot p (s) \, \rho_2 \, v (\rho_2)}{\dot p (s) + v (\rho_2)}
    \right)
    \\
    & &
    -
    \left(
      \chi\left(x-p (t)\right) - \chi\left(x-p (s)\right)
    \right)
    \left(\rho_1 \, v (\rho_1) - \rho_2 \, v (\rho_2)\right)
  \end{eqnarray*}
  To bound the absolute values of the terms above,
  use~\textbf{($\boldsymbol{\chi}$)}, the Lipschitz continuity of the
  map $\rho \to (\rho \, v (r)) / \left(\dot p + v (\rho)\right)$ with
  Lipschitz constant $L$, the boundedness of $\dot p$ and the
  Lipschitz continuity of the map $\rho \to \rho\, v (\rho)$ with
  Lipschitz constant $\Lip (\rho v)$
  \begin{eqnarray*}
    \modulo{\chi\left(x-p (t)\right) - \chi\left(y-p (s)\right)}
    & \leq &
    \Lip (\chi) \left(\modulo{x-y} + \modulo{p (t) - p (s)}\right)
    \\
    & \leq &
    \Lip (\chi)\left(1+\Lip (p)\right)
    \left(\modulo{x-y} + \modulo{t - s}\right)
    \\
    \modulo{\frac{\dot p (t) \, \rho_1 \, v (\rho_1)}{\dot p (t) + v (\rho_1)}
      -
      \frac{\dot p (t) \, \rho_2 \, v (\rho_2)}{\dot p (t) + v (\rho_2)}}
    & \leq &
    \Lip (p) \, L \, \modulo{\rho_1 - \rho_2}
    \\
    \modulo{\rho_1 \, v (\rho_1) - \rho_2 \, v (\rho_2)}
    & \leq &
    \Lip (\rho v) \, \modulo{\rho_1 - \rho_2}
  \end{eqnarray*}
  while the remaining term is estimated by means of $g$ as
  defined at~\eqref{eq:g} in Lemma~\ref{lem:speriamo}:
  \begin{eqnarray*}
    & &
    \modulo{  \frac{\dot p (t) \, \rho_1 \, v (\rho_1)}{\dot p (t) + v (\rho_1)}
      -
      \frac{\dot p (t) \, \rho_2 \, v (\rho_2)}{\dot p (t) + v (\rho_2)}
      -
      \frac{\dot p (s) \, \rho_1 \, v (\rho_1)}{\dot p (s) + v (\rho_1)}
      +
      \frac{\dot p (s) \, \rho_2 \, v (\rho_2)}{\dot p (s) + v (\rho_2)}}
    \\
    & = &
    \modulo{
      g\left(\rho_1, \dot p (t)\right)
      -
      g\left(\rho_2, \dot p (t)\right)
      -
      g\left(\rho_1, \dot p (s)\right)
      +g\left(\rho_2, \dot p (s)\right)
    }
    \\
    & \leq &
    \Lip (g)
    \modulo{\dot p (t) - \dot p (s)} \, \modulo{\rho_1 - \rho_2} \,.
    \\
    & \leq &
    \Lip (g) \, \Lip (\dot p) \,
    \modulo{t - s} \, \modulo{\rho_1 - \rho_2} \,.
  \end{eqnarray*}
  where we used the Lipschitz regularity of $t \to p (t)$. Going back
  to $\bar I_5$:
  \begin{eqnarray*}
    \bar I_5
    & \leq &
    C
    \left(\modulo{x-y} + \modulo{t-s}\right) \;
    \modulo{\rho_1 (t,x) - \rho_2 (s,y)} \;
    \phi\left(\frac{t+s}{2}, \frac{x+y}{2}\right) \;
    \delta_r\left(\frac{t-s}{2}\right)
  \end{eqnarray*}
  where
  \begin{equation}
    \label{eq:C}
    C
    =
    \Lip (\chi) \left(1+\Lip (p)\right)
    \left(\Lip (p) L + \Lip (\rho v)\right)
    +
    \Lip (g) \Lip (\dot p)
  \end{equation}
  and
  \begin{eqnarray*}
    & &
    \int_{\reali^+} \! \int_{\reali} \! \int_{\reali^+} \! \int_{\reali}
    \bar I_5 \, \partial_x \omega_r\left(\frac{x-y}{2}\right)
    \d{x} \d{t} \d{y} \d{s}
    \\
    & \leq &
    C
    \int_{\reali^+} \! \int_{\reali} \! \int_{\reali^+} \! \int_{\reali}
    \left(\modulo{x-y} + \modulo{t-s}\right)
    \modulo{\rho_1 (t,x) - \rho_2 (s,y)}
    \\
    & &
    \qquad\qquad\qquad
    \phi\left(\frac{t+s}{2}, \frac{x+y}{2}\right) \;
    \delta_r\left(\frac{t-s}{2}\right)
    \partial_x \omega_r\left(\frac{x-y}{2}\right)
    \d{x} \d{t} \d{y} \d{s}
    \\
    & \leq &
    C
    \int_{\reali^+} \! \int_{\reali} \! \int_{\reali^+} \! \int_{\reali}
    \frac{\modulo{x-y} + \modulo{t-s}}{r}
    \modulo{\rho_1 (t,x) - \rho_2 (s,y)}
    \\
    & &
    \qquad\qquad\qquad
    \phi\left(\frac{t+s}{2}, \frac{x+y}{2}\right) \;
    \delta_r\left(\frac{t-s}{2}\right)
    \frac{1}{r} \boldsymbol{1}_{[-r,r]} (x-y) \max \modulo{\delta'}
    \d{x} \d{t} \d{y} \d{s}
  \end{eqnarray*}
  so that
  \begin{equation}
    \label{eq:sudata}
    \lim_{r\to0}
    \int_{\reali^+} \! \int_{\reali} \! \int_{\reali^+} \! \int_{\reali}
    \bar I_5 \, \partial_x \omega_r\left(\frac{x-y}{2}\right)
    \d{x} \d{t} \d{y} \d{s}
    =
    C
    \int_{\reali^+} \! \int_{\reali}
    \modulo{\rho_1 (t,x) - \rho_2 (t,x)} \d{x} \d{t}
  \end{equation}
  The other terms $\bar I_0, \bar I_1, \bar I_3$ in~\eqref{eq:3.38}
  are estimated exactly as in~\cite[Formul\ae~(3.40), (3.41)
  and~(3.43)]{KarlsenRisebro2003}. Therefore, we have
  \begin{eqnarray*}
    & &
    \lim_{r\to 0+}
    \int_{\reali^+} \!\int_{\reali}  \!\int_{\reali^+} \!\int_{\reali}
    \bar I_0 \;
    \omega_r\left(\frac{x-y}{2}\right) \,
    \delta_r\left(\frac{t-s}{2}\right) \,
    \d{x} \d{t} \d{y} \d{s}
    \\
    & = &
    \int_{\reali^+} \! \int_{\reali}
    \modulo{\rho_1 (t,x) - \rho_2 (t,x)}
    \partial_t\phi (t,x)
    \d{x} \d{t} \,.
    \\
    & &
    \lim_{r\to 0+}
    \int_{\reali^+} \!\int_{\reali}  \!\int_{\reali^+} \!\int_{\reali}
    \bar I_1 \;
    \omega_r\left(\frac{x-y}{2}\right) \,
    \delta_r\left(\frac{t-s}{2}\right) \,
    \d{x} \d{t} \d{y} \d{s}
    \\
    & = &
    \int_{\reali^+} \! \int_{\reali}
    \sgn\left[\rho_1 (t,x) - \rho_2 (t,x)\right]
    \left(
      f\left(t,x,\rho_1 (t,x)\right)
      -
      f\left(t,x,\rho_2 (t,x)\right)
    \right)
    \partial_x\phi (t,x)
    \d{x} \d{t} \,.
    \\
    & &
    \lim_{r\to 0+}
    \int_{\reali^+} \!\int_{\reali}  \!\int_{\reali^+} \!\int_{\reali}
    \bar I_3 \;
    \omega_r\left(\frac{x-y}{2}\right) \,
    \delta_r\left(\frac{t-s}{2}\right) \,
    \d{x} \d{t} \d{y} \d{s}
    \\
    & = &
    \int_{\reali^+} \! \int_{\reali}
    \sgn\left[\rho_1 (t,x) - \rho_2 (t,x)\right]
    \left(
      \partial_xf\left(t,x,\rho_1 (t,x)\right)
      -
      \partial_xf\left(t,x,\rho_2 (t,x)\right)
    \right)
    \phi (t,x)
    \d{x} \d{t} \,.
    \\
  \end{eqnarray*}
  We now closely follow~\cite[Proof of
  Theorem~1.1]{KarlsenRisebro2003}. The latter relations, inserted
  in~\eqref{eq:3.38} together with~\eqref{eq:sudata}, yield
  \begin{eqnarray*}
    & &
    -
    \int_{\reali^+} \! \int_{\reali}
    \Bigl(
    \modulo{\rho_1 (t,x) - \rho_2 (t,x)}
    \partial_t\phi (t,x)
    \\
    & &
    \qquad
    +
    \sgn\left[\rho_1 (t,x) - \rho_2 (t,x)\right]
    \left(
      f\left(t,x,\rho_1 (t,x)\right)
      -
      f\left(t,x,\rho_2 (t,x)\right)
    \right)
    \partial_x\phi (t,x)
    \\
    & &
    \qquad
    +
    \sgn\left[\rho_1 (t,x) - \rho_2 (t,x)\right]
    \left(
      \partial_xf\left(t,x,\rho_1 (t,x)\right)
      -
      \partial_xf\left(t,x,\rho_2 (t,x)\right)
    \right)
    \phi (t,x)
    \Bigr)
    \d{x} \d{t}
    \\
    & \leq &
    C
    \int_{\reali^+} \! \int_{\reali}
    \modulo{\rho_1 (t,x) - \rho_2 (t,x)} \d{x} \d{t}
  \end{eqnarray*}
  for any test function $\phi \in
  \Cc\infty$. By~\cite[Lemma~3]{Kruzkov}, the map
  \begin{displaymath}
    (\rho_1, \rho_2)
    \to \sgn(\rho_1 - \rho_2 ) \left(
      \partial_xf(t,x,\rho_1) -
      \partial_xf(t,x,\rho_2) \right)
  \end{displaymath}
  is Lipschitz continuous, hence, possibly renaming the constant,
  \begin{eqnarray*}
    & &
    -
    \int_{\reali^+} \! \int_{\reali}
    \Bigl(
    \modulo{\rho_1 (t,x) - \rho_2 (t,x)}
    \partial_t\phi (t,x)
    \\
    & &
    \qquad
    +
    \sgn\left[\rho_1 (t,x) - \rho_2 (t,x)\right]
    \left(
      f\left(t,x,\rho_1 (t,x)\right)
      -
      f\left(t,x,\rho_2 (t,x)\right)
    \right)
    \partial_x\phi (t,x)
    \Bigr)
    \d{x} \d{t}
    \\
    & \leq &
    C
    \int_{\reali^+} \! \int_{\reali}
    \modulo{\rho_1 (t,x) - \rho_2 (t,x)} \d{x} \d{t}
  \end{eqnarray*}
  Choose arbitrary $t_1,t_2$ in $\left]0, T\right[$ with $t_1 < t_2$
  and the test function
  \begin{displaymath}
    \phi (t,x)
    =
    \int_{-\infty}^t \left(\delta_r (\tau-t_1) - \delta_r (t_2-\tau)\right) \d\tau
    \;
    \int_{-R}^R \delta_r\left(\modulo{x-y}\right) \d{y}
  \end{displaymath}
  with $\delta_r$ as in~\eqref{eq:deltar} and in the limit $R \to
  +\infty$ and $r \to 0$, obtain
  \begin{displaymath}
    \norma{\rho_1 (t_2) - \rho_2 (t_2)}_{\L1 (\reali;\reali)}
    \leq
    \norma{\rho_1 (t_1) - \rho_2 (t_1)}_{\L1 (\reali;\reali)}
    +
    C
    \int_{t_1}^{t_2} \norma{\rho_1 (\tau) - \rho_2 (\tau)}_{\L1 (\reali;\reali)} \d\tau \,.
  \end{displaymath}
  An application of Gronwall Lemma allows to conclude the proof,
  exactly as in~\cite[Proof of Theorem~1.1]{KarlsenRisebro2003}.
\end{proofof}

\noindent\textbf{Acknowledgment:} The authors thank Boris Andreianov
for useful discussions. Support by the INdAM--GNAMPA 2014 project
\emph{Conservation laws in the modelling of collective phenomena}.

{\small{

    \bibliographystyle{abbrv}

    \bibliography{kru_r}

  }}

\end{document}